\documentclass[11pt,oneside]{amsart}
\usepackage{standalone} 
\usepackage[T1]{fontenc}
\usepackage[latin9]{inputenc}
\usepackage{geometry}
\usepackage{verbatim}
\usepackage[all]{xy} 
\usepackage{textcomp}
\usepackage{dsfont}
\usepackage{caption}
\usepackage{amsthm}
\usepackage{amstext}
\usepackage{enumerate}
\usepackage{stmaryrd}
\usepackage{amssymb}
\usepackage{latexsym}
\usepackage{mathtools}
\usepackage{esint}
\usepackage{xcolor}
\usepackage{mathrsfs}
\usepackage{wasysym}
\usepackage{graphicx}
\usepackage{nomencl}
\usepackage[english]{babel}
\usepackage{tikz-cd}
\usepackage{pdfpages}
\usepackage{hyperref}
\usepackage[shortlabels]{enumitem}

\usepackage{bm}

\def\bR{{\mathbb R}}
\def\bA{{\mathbb A}}
\def\bR{{\mathbb R}}
\def\bZ{{\mathbb Z}}

\def\bT{{\mathbb T}}

\def\bZ{{\mathbb{Z}}}
\def\bN{{\mathbb{N}}}

\def\cA{{\mathcal A}}
\def\cF{{\mathcal F}}

\def\cM{{\mathcal M}}
\def\cN{{\mathcal N}}
\def\cI{{\mathcal I}}
\def\cC{{\mathcal C}}

\def\cE{{\mathcal E}}
\def\cR{{\mathcal R}}
\def\cL{{\mathcal L}}
\def\cQ{{\mathcal Q}}

\def\g{{\mathfrak g}}
\def\a{{\mathfrak a}}
\def\h{{\mathfrak h}}

\def\ker{{\text{ker}}}

\def\lb{{\llbracket}}
\def\rb{{\rrbracket}}

\def\ov\g{{\overline{\mathfrak{g}}}}
\def\TT{\mathbb{T}}

\def\Lie{\mathscr{L}}

\def\Fun{C_{\mathrm{poly}}(\cM)}

\def\Der{\mathrm{Der}}

\def\dCE{d_{\scriptscriptstyle{CE}}}

\def\dE{d_{\scriptscriptstyle{\cE}} }
\def\dR{d_{\scriptscriptstyle{\cR}} }
\def\rhocE{\boldsymbol{\rho}_{\scriptscriptstyle{\cE}}}
\def\scrsE{\scriptscriptstyle{\cE}}

\newcommand{\SP} [1]     {{\left\langle {{#1}} \right\rangle}}

\newcommand{\Cour}[1]      {[\![#1]\!]}
\newcommand{\m}[1]{{ \lvert #1 \rvert}}

\makeatletter
\newcommand{\iddots}{\mathinner{\mkern1mu\raise\p@%
  \hbox{.}\mkern2mu\raise4\p@\hbox{.}\mkern2mu%
  \raise7\p@\hbox{.}\mkern1mu}}
\makeatother

\makeatother

\newtheorem{theorem}{Theorem}[section]
\newtheorem*{theorem*}{Theorem}
\newtheorem{proposition}[theorem]{Proposition}
\newtheorem{corollary}[theorem]{Corollary}

\theoremstyle{definition}
\newtheorem{definition}[theorem]{Definition}

\theoremstyle{definition}
\newtheorem{remark}[theorem]{Remark}

\theoremstyle{definition}
\newtheorem{example}[theorem]{Example}

\title{Homological reduction of Courant algebroids}

\date{}
\author[1]{Henrique Bursztyn}
\address{Insituto de Matem\'atica Pura e Aplicada, Rio de Janeiro - Brazil}
\email{henrique@impa.br}
\author[2]{Pedro H. Carvalho}
\address{Faculty of Science, University of Hradec Kr\'alov\'e, Hradec Kr\'alov\'e - Czech Republic}
\email{pedro.silva@uhk.cz} 

\begin{document}

\begin{abstract}
We develop a homological model for hamiltonian reduction of Courant
algebroids, parallel to the BFV formulation of hamiltonian
reduction in symplectic geometry. To a hamiltonian action
on a Courant algebroid, we associate a differential
graded Courant--Dorfman algebra, obtained from the BFV algebra of the corresponding graded
symplectic manifold; under usual regularity assumptions, its cohomology recovers the Courant--Dorfman algebra of the reduced Courant algebroid.
\end{abstract}

\maketitle

\tableofcontents

\section{Introduction}


Courant algebroids lie at the intersection of several areas of geometry
and mathematical physics, including Poisson geometry,
generalized geometry, T-duality, field theory and sigma models, and
graded and shifted symplectic geometry, see e.g. \cite{AlekseevStrobl,BaragliaHekmati,Bressler,GarciaFernandezRicci,Gualtieri GC,Pym-Safronov,Roytenberg - AKSZ,SeveraValachSupergravity} and references therein. A foundational construction
in this setting is {\em hamiltonian Courant reduction} \cite{GGGR,Bursztyn}. In analogy with
hamiltonian reduction in symplectic geometry, symmetry and moment-map
data are used to restrict a Courant algebroid to a level set and descend
it to a quotient space.

Hamiltonian reduction in symplectic and Poisson geometry admits a
homological formulation through the BFV formalism, originating in physics \cite{BV, BRS, BF} and developed mathematically in \cite{Kimura,Kostant,Stasheff}. A hamiltonian action gives rise to a
differential graded Poisson algebra (or, more generally, a
$P_\infty$-algebra \cite{Carvalho}) whose cohomology recovers, under suitable
regularity assumptions, the Poisson algebra of functions on the reduced
space. In this formulation, the  infinitesimal symmetries,
moment map and constraints are encoded in a single differential-graded object.

In this paper we develop a parallel homological description
of hamiltonian Courant reduction. For this purpose, we introduce
{\em differential graded (dg) Courant--Dorfman algebras}, which are dg versions of
the Courant--Dorfman algebras of \cite{Roytenberg CD}; in this context the Courant-Dorfman structural maps (pairing, derivation and Leibniz bracket) are defined on a dg module $(\cE,\dE)$ over a commutative dg algebra $(\cR, \dR)$. A basic feature of a dg
Courant--Dorfman algebra of degree $-m$ is that it induces an ordinary
Courant--Dorfman algebra on the cohomology pair
$$
(H^0_\cR,H^m_\cE).
$$
We provide natural examples of dg Courant-Dorfman algebras associated with
commutative dg algebras, which are algebraic analogues of exact
Courant algebroids.

Our main result, Theorem~\ref{thm:main}, associates with every hamiltonian action on a Courant algebroid $E\to M$ a
nondegenerate dg Courant--Dorfman algebra $(\cR,\cE)$ of degree $-1$. If the action
satisfies the regularity assumptions of hamiltonian Courant reduction (Theorem~\ref{thm:geometricCred}),
there is a natural isomorphism
\[
(H^0_\cR,H^1_\cE)
\cong
\bigl(C^\infty(M_{\mathrm{red}}),\Gamma(E_{\mathrm{red}})\bigr)
\]
of Courant--Dorfman algebras. We therefore regard $(\cR,\cE)$ as a homological
model for the reduced Courant algebroid.

The parallel with the classical theory is summarized as follows.

\medskip

\begin{small}
\[
\begin{array}{c|c|c}
\text{reduction data}
&
\text{homological model}
&
\text{structure recovered in cohomology}
\\ \hline
\text{Hamiltonian symplectic action}
&
\text{dg Poisson algebra}
&
\text{Poisson algebra of the reduced space}
\\[1mm]
\text{Hamiltonian Courant action}
&
\text{dg Courant--Dorfman algebra}
&
\begin{gathered}
\text{Courant--Dorfman algebra}\\[-1ex]
\text{of the reduced Courant algebroid}
\end{gathered}
\end{array}
\]
\end{small}

\medskip

The construction has an algebraic core that is independent of
Courant reduction. We show that a bigraded Poisson algebra $\cA$ of
bidegree $(-2,-2)$, equipped with a charge
\[
S=Q+\Theta,
\qquad
\{S,S\}=0,
\]
with $Q\in\cA^{2,3}$ and $\Theta\in\cA^{3,3}$ (see  Definition~\ref{def:bigHPA}), determines a dg Courant--Dorfman  algebra  $(\cR, \cE)$ of degree $-1$,
where $\cR=\cA^{0,\bullet}$ and  $\cE=\cA^{1,\bullet}$. Its structural maps are obtained by
\[
d_\cR r=\{Q,r\},
\qquad
d_\cE e=\{Q,e\},
\qquad
\langle e_1,e_2\rangle=\{e_1,e_2\},
\]
\[
\partial r=\{\Theta,r\},
\qquad
[e_1,e_2]
=
-(-1)^{|e_1|}
\bigl\{\{\Theta,e_1\},e_2\bigr\},
\]
for $r\in \cR$, $e,e_1,e_2 \in \cE$, see Thm~\ref{thm: derived dg CD}. The dg Courant--Dorfman identities follow from the Poisson identities
and the master equation. 
This provides a differential graded extension of the
derived-bracket construction relating Courant algebroids to symplectic
$\bN Q$-manifolds of degree $2$ (cf. \eqref{eq: Courant str RS corresp}).

To apply this algebraic construction to Courant reduction,
we use of the graded-symplectic description of Courant algebroids. A
Courant algebroid is equivalent to a symplectic
$\bN$-manifold of degree $2$ equipped with a self-commuting degree $3$ hamiltonian \cite{Roytenberg}; under this correspondence, a
hamiltonian action on the Courant algebroid becomes a hamiltonian action
of a differential graded Lie algebra (dgla) on the associated graded
symplectic manifold \cite[\S~8]{GGGR}.
The BFV construction for such an action produces a graded Poisson
algebra $\cA_{\scriptscriptstyle{BFV}}$. We endow this algebra with a
second grading, called {\em weight}, for which it becomes a bigraded
Poisson algebra of bidegree $(-2,-2)$ with charge
$S=Q+\Theta$
as above. Here $Q$ is the BRST charge determined by the reduction data,
while $\Theta$ combines the Courant hamiltonian with the dgla
differential. The general construction then yields the dg
Courant--Dorfman algebra $(\cR,\cE)$ from the weight-$0$ and weight-$1$
components of $\cA_{\scriptscriptstyle{BFV}}$.

In the exact case, the BFV algebra
admits a description as a shifted-cotangent bundle, and the resulting dg
Courant--Dorfman algebra becomes the ``standard'' Courant
algebroid of a dg manifold. 
In a special case, this viewpoint relates the BFV algebra $\cA_{\scriptscriptstyle{BFV}}$ to the Kalkman
BRST model for equivariant cohomology \cite{Kalkman}.

In forthcoming work, we will incorporate the reduction of Dirac and generalized complex strutures into this homological setting. More broadly, the framework developed here admits natural extensions from differential graded structures to their $\infty$-analogues as well as their global counterparts (in the spirit of \cite{CDSZ,Nikolay}).

\smallskip

\noindent{\bf Structure of the paper}. Section~\ref{sec:dgCD} introduces dg
Courant--Dorfman algebras, their cohomology, and presents standard and twisted
examples. Section~\ref{sec:Cred} reviews hamiltonian dgla actions on
Courant algebroids and Courant reduction, following
\cite[\S~8]{GGGR}, and states Theorem~\ref{thm:main}.

Section~\ref{sec:hamPA} develops the general construction from
bigraded hamiltonian Poisson algebras to dg Courant--Dorfman algebras.
Section~\ref{sec:graded} recalls the graded-symplectic description of
Courant algebroids and their hamiltonian actions.
Section~\ref{sec:BFV} constructs the associated BFV algebra;  Section~\ref{sec:proof} places
it in the framework of Section~\ref{sec:hamPA} and proves
Theorem~\ref{thm:main}.

Section~\ref{sec:interpret} gives a graded-geometric
interpretation of the construction, including the shifted-cotangent
description in the exact case and its relation with the Kalkman BRST model.

\medskip

\paragraph{{\bf Acknowledgments.}}
H. B. thanks FAPERJ (CNE) and CNPq for financial support. P. H. C. was supported by the
Czech Science Foundation (GA\v{C}R), grant no.~24-10031K. This work grew
out of the second author's PhD thesis \cite{thesis}, 
supported by CNPq. We are grateful to D. Alvarez, A. Cabrera, M. Cueca, and A. Kotov for helpful discussions.

\section{Dg Courant-Dorfman algebras}\label{sec:dgCD}

We introduce a differential-graded version of the Courant-Dorfman algebras of \cite{Roytenberg CD}, whose role in homological Courant reduction is parallel to that of differential graded Poisson algebras in the classical homological reduction of Poisson structures \cite{Kostant,Stasheff}.

Throughout the paper, we use $\bR$ as the ground field for (graded) vector spaces and algebras.


\subsection{Definition and shifts}
Let $(\cR, \dR)$ be a commutative differential graded (dg) algebra, that is, a graded vector space $\cR = \oplus_{k \in \bZ} \cR^k$  with a graded commutative associative product, 
$$
r_1 r_2 = (-1)^{\m{r_1} \m{r_2}} r_2 r_1,
$$  
where $\m{r_i}$ denotes the degree of a homogeneous element $r_i$,
and a differential $\dR: \cR^\bullet \rightarrow \cR^{\bullet+1}$ satisfying 
\begin{equation}\label{eq:leib1}
\dR(r_1 r_2) = \dR(r_1) r_2 + (-1)^{\m{r_1}} r_1 \dR(r_2).
\end{equation}
A dg module $(\cE, \dE)$ over $(\cR, \dR)$ is graded vector space $\cE = \oplus_{\ell \in \bZ} \cE^\ell$ with an (left) action of  $\cR,$ 
$$
\cR^k \otimes \cE^\ell \rightarrow \cE^{k + \ell}, \quad\; r\otimes e \longmapsto re, 
$$
and a differential $\dE : \cE^\bullet \rightarrow \cE^{\bullet + 1},$
such that
\begin{equation}\label{eq:leib2}
\dE(re)=\dR(r)e+ (-1)^{|r|}r \, \dE(e).
\end{equation}
We note that $\cE$ also carries a right-module structure via $er:=(-1)^{\m{r}\m{e}} r e$.

If $\cE$ is a dg $\cR$-module, then the space of graded linear homomorphisms
$
\cE^\vee:= \mathrm{Hom}_{\cR}(\cE,\cR)
$
is also a dg $\cR$-module with differential $d_{\cE^\vee}$ given by 
\begin{equation}\label{eq:dualdiff}
(d_{\cE^ \vee}\xi)(e) := d_\cR(\xi(e))- (-1)^{\m{\xi}}\xi(d_\cE e).
\end{equation}

\begin{definition}\label{def:dgDC}
     A {\em differential graded (dg) Courant-Dorfman algebra} $(\cR,\cE)$ of degree $-m$ consists of a dg-algebra $(\cR, d_{\scriptscriptstyle{\cR}})$ and a dg-module $(\cE, d_{\scriptscriptstyle{\cE}})$ over $\cR$ together with
    \begin{itemize}      
        \item[(a)] an $\cR$-bilinear form  $\langle \cdot, \cdot \rangle: \cE \otimes_{\cR} \cE \rightarrow \cR$ of degree $-2m$, 
        $$
        \SP{\cE^k,\cE^l}\subseteq \cR^{k+l-2m}, \quad \mathrm{and} \quad  \SP{re_1,e_2} =  r \SP{e_1,e_2}, \;\;\; \SP{e_1,re_2}=(-1)^{r \m{e_1}}r \SP{e_1,e_2},
        $$
satisfying 
$$
\SP{e_1,e_2}=(-1)^{|e_1||e_2|+m}\SP{e_2,e_1},
$$     
        \item[(b)] a derivation $\partial: \cR \rightarrow \cE$ of degree $m$, i.e., $\partial(\cR^k)\subseteq \cE^{k+m}$ and
        $$
        \partial(r_1 r_2) = \partial(r_1)r_2  + (-1)^{m \m{r_1}}r_1 \partial(r_2),
        $$
        \item[(c)] a Leibniz bracket $[\cdot, \cdot]: \cE \otimes_{\cR} \cE \rightarrow \cE$ of degree $-m$: $[\cE^k, \cE^l]\subseteq \cE^{k+l-m}$ and 
        $$
        [e_1, [e_2,e_3]] = [[e_1, e_2], e_3] + (-1)^{(|e_1| - m)(|e_2|- m)}[ e_2, [e_1,e_3]], 
        $$        
    \end{itemize}
    satisfying the following conditions: for  $e, e_1, e_2, e_3 \in \cE$ and $r, r_1, r_2 \in \cR,$
    \begin{enumerate}[label=(\roman*.)]
        \item $\dE [e_1, e_2] = [\dE e_1 ,  e_2] + (-1)^{|e_1| - m } [e_1, \dE e_2]$;
        \item the map 
$\flat: \cE\longrightarrow \cE^\vee$,
$e\longmapsto \SP{e,\cdot}$,
is a morphism of dg $\cR$-modules; equivalently,
$$
\dR \langle e_1, e_2 \rangle = \langle \dE e_1, e_2 \rangle + (-1)^{|e_1|} \langle e_1, \dE e_2 \rangle;$$ 
        
        \item  $\dE \circ \partial = (-1)^m \partial \circ \dR$, i.e., $\partial$ is a chain map of degree $m$,
     \end{enumerate}   
  and     
      \begin{enumerate}  
        
        \item $\langle e_1, \partial \langle e_2, e_3 \rangle  \rangle = \langle [e_1, e_2], e_3 \rangle + (-1)^{(|e_1| - m) |e_2|} \langle e_2 , [e_1, e_3] \rangle$, 
        
        \item $[e_1, e_2] + (-1)^{(|e_1| - m)|(|e_2| -m)} [e_2, e_1] = (-1)^{m(|e_1|-m)}\partial \langle e_1, e_2 \rangle 
        = (-1)^{m(|e_1|-1)}\partial \langle e_1, e_2 \rangle$, %
         \item $[\partial r, e] = 0$,
        
        \item $[e_1, re_2] = \langle e_1, \partial r\rangle e_2 + (-1)^{(|e_1| - m)|r|}r[e_1, e_2]$,

        \item $\langle \partial r_1, \partial r_2  \rangle = 0$.
    \end{enumerate}
\hfill $\diamond$
\end{definition}

When  $\cR$ is concentrated in degree $0$ and $\cE$ is concentrated in degree $m$, one recovers the usual definition of a Courant-Dorfman algebra from \cite{Roytenberg CD}.

The next result explains how to change the degree of a dg Courant-Dorfman algebra by means of shifts.
Let $(\cR,\cE)$ be a dg Courant-Dorfman algebra of degree $0$,
and consider the shifted space
$\cF:=\cE[-m]$, so that the shift map $s:\cE\to\cF$ has degree $m$.
One may regard $\cF$ as a dg $\cR$-module by keeping the
$\cR$-module structure and differential unchanged (i.e.,
$r s(e)=s(re)$, and
$d_\cF\circ s=s\circ d_\cE$).

\begin{proposition}\label{prop:shifts}
The pair $(\cR,\cF)$, endowed with  the bracket, derivation and pairing given by
\begin{equation}\label{eq:shiftedCD}
[s(e_1),s(e_2)]_{\scriptscriptstyle{\cF}}
:=s([e_1,e_2]),
\quad
\partial_{\scriptscriptstyle{\cF}}r
:=(-1)^{m|r|}s(\partial r),
\quad
\SP{s(e_1),s(e_2)}_{\scriptscriptstyle{\cF}}
:=(-1)^{m|e_2|}\SP{e_1,e_2}
\end{equation}
is a dg Courant-Dorfman algebra of degree $-m$.
\end{proposition}

\begin{remark}
Alternatively, we may regard $\cF$ as a dg $\cR$-module for which
the shift map $s:\cE\to\cF$ is an isomorphism of dg $\cR$-modules
of degree $m$ (i.e., $r s(e)=(-1)^{m|r|}s(re)$, and
$d_\cF\circ s=(-1)^m s\circ d_\cE$),
in which case a dg Courant-Dorfman algebra structure of degree $-m$
on $(\cR,\cF)$ is defined by
\[
[s(e_1),s(e_2)]_{\scriptscriptstyle{\cF}}
:=s([e_1,e_2]),
\qquad
\partial_{\scriptscriptstyle{\cF}}
:=s\circ\partial,
\qquad
\SP{s(e_1),s(e_2)}_{\scriptscriptstyle{\cF}}
:=(-1)^{m|e_1|}\SP{e_1,e_2}.
\]
The degree $-m$ Courant-Dorfman algebra of Prop.~\ref{prop:shifts} is isomorphic to this one
via
\[
\cF\longrightarrow\cF,
\qquad
f\longmapsto(-1)^{m|f|}f.
\]

\hfill $\diamond$
\end{remark}

The following examples of dg Courant-Dorfman algebras are described in detail 
in App.~\ref{app:twisted}, including basic definitions and sign conventions.

\begin{example}[Standard]\label{ex:standard} Let $(\cR, \dR)$ be a commutative dg algebra. Let $\Der(\cR)$ be the dg $\cR$-module of graded derivations, with differential $d_\Der= [\dR, \cdot]$ defined by the graded commutator of derivations. Let $(\Omega^1_\cR= \Der(\cR)^\vee, d_\Omega)$ be the dual dg-module, with differential \eqref{eq:dualdiff} (for our purposes, there will be no need for Kahler differentials).

The \emph{standard Courant-Dorfman algebra} (of degree $0$) is $(\cR, \bT_\cR)$, where the dg $\cR$-module is
$$
\mathbb{T}_\cR:= \Der(\cR)\oplus \Omega^1_\cR, 
$$
with differential
$$
\qquad d_{\bT}(X,\alpha)=(d_\Der(X), d_\Omega(\alpha)),
$$
equipped with the pairing 
\begin{equation}\label{eq:stdpair}
\SP{e_1,e_2} = i_X\beta + (-1)^{\m{e_1}\m{e_2}}i_Y\alpha,
\end{equation}
for $e_1=(X,\alpha)$, $e_2=(Y,\beta)$, the
bracket
\begin{equation*}
\Cour{(X,\alpha), (Y,\beta)}=([X,Y], \Lie_X\beta - (-1)^{\m{e_1}\m{e_2}}i_Yd\alpha),\label{eq:brktsdt exm}
\end{equation*}
and derivation $\partial: \cR \to \bT_\cR$ given by
\begin{equation}\label{eq:stdpartial}
\partial r = (0, dr). 
\end{equation}

Here $i_X$, $\Lie_X$ and $d$ are the usual (graded) Cartan operators, see $\S$ \ref{subsec:dgcartan}.
\hfill $\diamond$
\end{example}

\begin{example}[Twists]\label{ex:twists}
Let $\Omega^k_\cR = \mathrm{Hom}_\cR(\wedge^k \mathrm{Der}(\cR), \cR)$ be the graded $\cR$-module of $k$-forms, and denote by $\Omega_\cR^{k,\ell}$ the subspace of $k$-forms of internal degree $\ell$.

Let $F\in \Omega_\cR^{2,1}$ and $H\in \Omega_\cR^{3,0}$ be such that
\begin{equation}\label{eq:FH}
d_\Omega F=0, \quad dH=0, \quad d_\Omega H=dF.
\end{equation}
Then $d_{\bT,F}: \bT_\cR\to \bT_\cR$,
\begin{equation}\label{eq:Fdiff}
d_{\bT,F}(X,\alpha) := (d_\Der(X), d_\Omega(\alpha) + (-1)^{\m{X}}i_XF),
\end{equation}
is a differential that makes $\bT_\cR$ into a dg module over $\cR$ in such a way that the pairing \eqref{eq:stdpair}, derivation \eqref{eq:stdpartial} and bracket
\begin{equation}\label{eq:twisted-bracket}
\Cour{(X,\alpha), (Y,\beta)}_H=([X,Y], \Lie_X\beta - (-1)^{\m{e_1}\m{e_2}}i_Yd\alpha + (-1)^{\m{X}\m{Y}}i_Yi_XH)
\end{equation}
define a dg Courant-Dorfman algebra of degree $0$.

We call it the \emph{$(F,H)$-twisted dg Courant-Dorfman algebra} structure on $\bT_\cR$.
\hfill $\diamond$
\end{example}

\begin{example}[Shifts] \label{ex:shifted}

One can shift the previous example to obtain a dg Courant-Dorfman algebra structure of degree $-m$ on $(\cR, \bT_\cR[-m])$, following Prop.~\ref{prop:shifts}.

Consider $F \in \Omega^{2,1}_\cR$ and $H\in \Omega_\cR^{3,0}$ satisfying \eqref{eq:FH}. By shifting their degrees, they define elements $\bm F \in (\Omega_\cR^2[-2m])^{2m+1}$ and $\bm H \in (\Omega_\cR^3[-3m])^{3m}$.  In terms of the shifted Cartan operations 
$$
\bm{i}_X, \qquad \bm{\Lie}_X, \qquad \bm{d}
$$ 
defined in $\S$ \ref{subsection:shifts}, the shifted structures \eqref{eq:shiftedCD} are expressed as follows: the differential on $\cE= \bT_\cR[-m]=\Der(\cR)[-m]\oplus \Omega_\cR^1[-m]$ is
\begin{equation}\label{eq:dFt}
\dE(X,\alpha) = (d_\Der X , d_\Omega \alpha + (-1)^{\m{X}}\bm{i}_X \bm{F}),
\end{equation}
the pairing is 
\begin{equation}\label{eq:canpairm}
\SP{(X,\alpha),(Y,\beta)}=\bm{i}_X\beta+(-1)^{|e_1||e_2|+m}\bm{i}_Y\alpha,
\end{equation}
for homogeneous $e_1=(X,\alpha)$, $e_2=(Y,\beta)$, the degree $m$ derivation is
\begin{equation}\label{eq:dm}
\partial r=(0,\bm{d}r),
\end{equation}
and the Leibniz bracket is
\begin{equation}\label{eq:lm}
\Cour{(X,\alpha),(Y,\beta)}_H=(\bm{[}X,Y\bm{]},\bm{\Lie}_X\beta-(-1)^{(|e_1|-m)(|e_2|-m)}\bm{i}_Y\bm{d}\alpha + (-1)^{(\m{X}-m)\m{Y}}\bm{i}_Y\bm{i}_X \bm{H}).
\end{equation}

When $F=0$, $H=0$, we call it the {\em standard dg Courant-Dorfman algebra of degree $-m$}.

\hfill $\diamond$
\end{example}

\subsection{Passing to cohomology}

Similarly to other dg structures, passing to cohomology turns dg Courant-Dorfman algebras into ordinary Courant-Dorfman algebras.


\begin{proposition}\label{prop: CD algebra in cohomology} For a dg Courant-Dorfman algebra $(\cR,\cE)$ of degree $-m$, the pair $(H^0_{\cR},H^m_{\cE})$ inherits a Courant-Dorfman algebra structure. 
\end{proposition} 

\begin{proof} First note that $(\cR^0,\cE^m)$ is a Courant-Dorfman algebra. The associative product on $\cR^0$ descends to $H^0_{\cR}$ because $\dR$ is a derivation. Since $(\cE,\dE)$ is a dg module over $(\cR,\dR)$ and $\cE^m$ is a $\cR^0$-module, it follows that $H^m_{\cE}$ is an $H^0_{\cR}$-module. We now verify that the structural maps of the Courant-Dorfman algebra $(\cR^0,\cE^m)$ induce a Courant-Dorfman algebra structure on $(H^0_{\cR},H^m_{\cE})$. Let 
\[ 
Z^m_\cE\coloneqq \ker(\dE:\cE^m\rightarrow\cE^{m+1}), \qquad Z^0_\cR\coloneqq \ker(\dR:\cR^0\rightarrow\cR^1). 
\] 
From condition (i.) in Def.~\ref{def:dgDC}, it follows that
\[ 
[Z^m_\cE,Z^m_\cE]\subseteq Z^m_\cE, 
\] and that 
\[ [\dE(\cE^{m-1}),Z^m_\cE] + [Z^m_\cE,\dE(\cE^{m-1})] \subseteq \dE(\cE^{m-1}). 
\] 
Consequently, 
\begin{equation}\label{eq:bracketchom} [\cdot,\cdot]: H^m_{\cE}\otimes_{\bR}H^m_{\cE} \longrightarrow H^m_{\cE}, \qquad [\overline{e_1},\overline{e_2}] \coloneqq \overline{[e_1,e_2]}, 
\end{equation} 
is well-defined. On the other hand, condition (ii.) implies that 
\[ \SP{Z^m_\cE,Z^m_\cE}\subseteq Z^0_\cR \] 
and 
\[ \SP{\dE(\cE^{m-1}),Z^m_\cE} + \SP{Z^m_\cE,\dE(\cE^{m-1})} \subseteq \dR(\cR^{-1}), 
\] 
which ensures that 
\begin{equation}\label{eq:spchom} \SP{\cdot,\cdot}: H^m_{\cE}\otimes_{H^0_{\cR}}H^m_{\cE} \longrightarrow H^0_{\cR}, \qquad \SP{\overline{e_1},\overline{e_2}} \coloneqq \overline{\SP{e_1,e_2}}, 
\end{equation} 
is also well-defined. Since
\[ 
\dE\circ\partial=(-1)^m\partial\circ\dR, 
\] 
the map $\partial:\cR\rightarrow\cE$ also descends to cohomology, yielding \begin{equation}\label{eq:partialchom} \partial: H^0_{\cR}\longrightarrow H^m_{\cE}, \qquad \partial(\overline r) \coloneqq \overline{\partial r}. \end{equation} A routine verification shows that the structural maps \eqref{eq:bracketchom}, \eqref{eq:spchom}, and \eqref{eq:partialchom} endow $(H^0_{\cR},H^m_{\cE})$ with the structure of a Courant-Dorfman algebra. 
\end{proof}



\subsection{Nondegenerate pairings}

 If $(\cR,\cE)$ is a dg Courant-Dorfman algebra of degree $-m$, as in Def.~\ref{def:dgDC}, we say that the pairing $\SP{\cdot,\cdot}$ is (strongly) {\em nondegenerate} if
\[
\flat: \cE\longrightarrow \cE^\vee:= \mathrm{Hom}_{\cR}(\cE,\cR),
\qquad
e\longmapsto \SP{e,\cdot},
\]
is a bijection, and hence an isomorphism of dg $\cR$-modules (by (ii.)).
In such a case, the derivation $\partial$ is equivalent to a homogeneous map of degree $-m$,
$$
\rho:\cE\to \Der(\cR),
$$
defined by
$$
\rho(e)(r)=\SP{e,\partial r},
$$
referred to as the {\em anchor}. (Note that
$\rho(re)=r\rho(e)$, so $\rho$ is not $\cR$-linear in the graded sense). 
We can rewrite the various axioms in Def.~\ref{def:dgDC} in terms of $\rho$:

\begin{itemize}
\item Condition (iii.) says that $\rho: \cE\to \Der(\cR)$ intertwines differentials: 
\begin{equation}\label{eq:rhointertwine}
\rho\circ \dE =d_{\Der}\circ \rho.
\end{equation}

\item Conditions (1), (2) become, respectively,
 \begin{itemize}
\item $\rho(e_1)\SP{e_2,e_3} =  \langle [e_1, e_2], e_3 \rangle + (-1)^{(|e_1| - m) |e_2|} \langle e_2 , [e_1, e_3] \rangle$,
 \item
$ \SP{e, [e_1, e_2] + (-1)^{(|e_1| - m)|(|e_2| -m)} [e_2, e_1]} = (-1)^{m(|e_1|- m)}\rho(e)\SP{e_1, e_2},$ 
\end{itemize}

\item Assuming that (1) and (2) hold, condition (3) is equivalent to 
$$
\rho([e_1, e_2]) = [\rho(e_1), \rho(e_2)],
$$ 
for all $e_1,e_2\in \cE$. Condition (4) becomes
\begin{equation}\label{eq:CDleib}
[e_1, re_2] = (\rho(e_1) r )e_2 + (-1)^{(|e_1| - m)|r|}r[e_1, e_2],
\end{equation}
while (5) says that $\rho$ and $\partial$ satisfy $\rho\circ \partial=0$.
\end{itemize}

Assuming the nondegeneracy of the pairing, we can also simplify Def.~\ref{def:dgDC} by removing redundancies. One can check that

\begin{itemize}   
\item Condition (4) follows from (1) and (2), 

\item If the pairing is full (i.e., $\mathrm{span}_\cR\{\SP{e, e'},\; e,e'\in \cE \}=\cR$), condition (3) follows from (2) and the Jacobi-like identity in (c).

\item If $\cE$ is a faithful module (i.e., $r\cE=0$ implies that $r=0$, for all $r\in \cR$), then (5) follows from (3) and (4).

\end{itemize}  

\begin{example}[Dg Courant algebroids] A {\em Courant algebroid} \cite{LWX,Roytenberg} structure on a vector bundle $E\to M$ is the same as an ungraded (concentrated in degree $0$) Courant-Dorfman algebra with nondegenerate pairing on $(C^\infty(M),\Gamma(E))$.  

More generally, the main examples of dg Courant-Dorfman algebras in this paper will be defined on pairs 
$(\cR, \cE)$, where $\cR=C(\cM)$ is the algebra of function on a dg manifold $\cM$ and $\cE=\Gamma(\mathcal V)$ is the module of sections of a dg vector bundle $\mathcal{V}\to \cM$, and the pairing is nondegenerate. By analogy, we should view these examples as  {\em dg Courant algebroids}.
In these geometric examples, the pairings are full, and the modules are faithful.

We note that our notion of dg Courant algebroids does not exactly match the one in \cite{Vysoky};
for instance, in the latter the pairing and Leibniz bracket are required to have the same degree, whereas this is not the case in Def.~\ref{def:dgDC}.
\hfill $\diamond$
\end{example}

\section{Hamiltonian reduction of Courant algebroids}\label{sec:Cred}

In this section, we review the procedure of hamiltonian reduction of Courant algebroids  described in \cite[$\S$ 8.3]{GGGR}, extending the reduction of exact Courant algebroids from \cite{Bursztyn}.

To setup our notation, we consider Courant algebroids defined by a vector bundle $E\rightarrow M$ equipped with a pseudo-riemannian metric $\SP{\cdot,\cdot}$
(i.e.,  a fibrewise non-degenerate, symmetric bilinear pairing), anchor map $\rho:E\rightarrow TM$ and an bracket $\Cour{\cdot,\cdot}:\Gamma(E)\times \Gamma(E)\rightarrow \Gamma(E)$.

\begin{example}
A Courant algebroid over a point is a quadratic Lie algebra, i.e., a Lie algebra equipped a symmetric, nondegenerate and ad-invariant bilinear pairing. \hfill $\diamond$
\end{example}

\begin{example}\label{ex:exactCA}
Given a manifold $M$, any closed 3-form $\eta \in \Omega^ 3(M)$ defines a Courant algebroid $E=TM\oplus T^*M$ with fibrewise pairing $\SP{(X,\alpha),(Y,\beta)}=\beta(X)+\alpha(Y)$ and bracket 
$$
\Cour{(X,\alpha),(Y,\beta)}:= ([X,Y],\Lie_X\beta - i_Yd\alpha + i_Yi_X\eta)
$$
for $X,Y \in \mathfrak{X}(M)$ and $\alpha,\beta\in \Omega^1(M)$; see \cite{SeWe}.
Courant algebroids of this type are called {\em exact}. \hfill $\diamond$
\end{example}

Given a pseudo-riemannian vector bundle $(E,\SP{\cdot,\cdot})$, its infinitesimal automorphisms are codified by the {\em Atiyah algebroid} $\bA_E\to M$, whose sections are derivation of $E$ that preserve $\SP{\cdot,\cdot}$. I.e., each section of $\bA_E$ is a pair $(D,X)$, where $D:\Gamma(E)\to \Gamma(E)$ is an $\mathbb{R}$-linear map, $X\in \mathfrak{X}(M)$, satisfying
\begin{align*}
D(fe) & = (\Lie_X f)e + f D(e),\\
\Lie_X\SP{e_1,e_2} & = \SP{D(e_1),e_2} + \SP{e_1, D(e_2)},
\end{align*}
for $f\in C^\infty(M)$, $e,e_1,e_2 \in \Gamma(E)$. The Atiyah algebroid is a Lie algebroid with Lie bracket on $\Gamma(\bA_E)$ given by the commutator of derivations and anchor defined by the {\em symbol map}
$$
\sigma: \Gamma(\bA_E)\to \mathfrak{X}(M), \qquad (D,X)\mapsto X.
$$
The Lie algebra of infinitesimal automorphisms of a Courant algebroid $(E, \SP{\cdot,\cdot}, \rho, \Cour{\cdot,\cdot})$ consists of sections $(D,X)\in \Gamma(\bA_E)$ that  preserve the Courant bracket,
$$
D(\Cour{e_1,e_2}) = \Cour{D(e_1),e_2} + \Cour{e_1, D(e_2)}.
$$
This implies that the anchor map is also preserved: $\rho(D(e)) = [X,\rho(e)]$.
On a Courant algebroid, any $e\in \Gamma(E)$ defines an infinitesimal automorphism by 
$$
\mathrm{ad}(e) = \Cour{e,\cdot},
$$ 
with symbol $\rho(e)$.

\subsection{Hamiltonian actions on Courant algebroids}\label{subsection: hmlt actions Courant alg} 
Courant algebroid are naturally acted upon by differential graded (dg) Lie algebras  \cite[$\S$ 8]{GGGR}. Let $\g$, $\a$, and $\h$ be vector spaces, and consider the graded vector space $\ov\g:= \h[2] \oplus \a[1] \oplus \g$, so that elements of $\h$ have degree $-2$, elements of $\a$ have degree $-1$, and elements of $\g$ have degree $0$.

A differential graded Lie algebra (dgla) $(\ov\g, [\cdot,\cdot], \delta)$ can be equivalently expressed by the data $(\g, \a, \h, \delta, \varpi)$, where
\begin{itemize}
\item  $\g$ is a Lie algebra,  $\mathfrak{a}$ and $\h$ are $\g$-modules, and   $\varpi\colon \mathfrak{a} \otimes \mathfrak{a} \to \h$ is a $\g$-equivariant symmetric bilinear map;
\item $\delta$ denotes a sequence of  linear maps,
\begin{equation}\label{eq:deltaseg}
\h \stackrel{\delta}{\longrightarrow} \mathfrak{a} \stackrel{\delta}{\longrightarrow} \g, 
\end{equation}
defining a complex, i.e., $\delta^2=0$, preserving the $\g$-module structures ($\g$ is viewed as a $\g$-module with respect to the adjoint representation) and compatible with $\varpi\colon \mathfrak{a} \otimes \mathfrak{a} \to \h$  in the sense that 
\begin{equation}\label{eq:condsbeta}
\delta \varpi(v_1,v_2) = (\delta v_1) \cdot v_2 + (\delta v_2) \cdot v_1, \quad (\delta v) \cdot w = \varpi(v, \delta w),
\end{equation}
for $v$, $v_1$, $v_2 \in \mathfrak{a}$ and $w \in \h$. 

Whenever necessary for clarity, we use the notation
$$
\delta_\h = \delta|_\h, \qquad \delta_\a:=\delta|_\a.
$$
\end{itemize}

\begin{definition}\label{def:hamact}
A {\em hamiltonian action} of a dgla $(\g, \a, \h, \delta, \varpi)$ on a Courant algebroid $(E\to M, \SP{\cdot,\cdot}, \rho, \Cour{\cdot,\cdot})$ consists of a triple $(\varphi, \varrho, \mu),$ where
\begin{itemize} 
\item  $\varphi\colon \g \to \Gamma(\mathbb{A}_E)$ is a  Lie algebra morphism 
such that 
\begin{equation}\label{eq: varphi calg aut}
u\cdot \Cour{e_1,e_2} =\Cour{u\cdot e_1, e_2}+
\Cour{e_1, u\cdot e_2}, \qquad u\in \g  \quad  e_1, e_2\in \Gamma(E),
\end{equation} 
for $u\cdot e:= \varphi(u)(e)$, i.e.,
 $\varphi$ maps $\g$ to infinitesimal Courant automorphisms; in particular, $E$ is a $\g$-equivariant vector bundle, and we have an induced  $\g$-action on $M$ by
 \begin{equation}\label{eq:Mact}
 \underline{\varphi}:= \sigma\circ \varphi : \g \rightarrow \mathfrak{X}(M),
 \end{equation}
 where $\sigma$ is the symbol map;

\item  $\mu\colon M\to \h^*$  and $\varrho\colon \mathfrak{a}\to \Gamma(E)$  are $\g$-equivariant maps satisfying
\begin{equation}\label{eq:Kisot}
\mu^*(\varpi(v_1,v_2))=\SP{\varrho(v_1),\varrho(v_2)}, \quad v_1, v_2 \in \mathfrak{a};
\end{equation}

\item the maps $\varphi$, $\varrho$, $\mu$ are such the following diagram commutes:
\begin{equation}\label{diagram: hamiltonian act Courant}
    \begin{tikzcd}
        \mathfrak{h} \arrow[d, "\mu^*"] \arrow[r, "\delta"] 
        & \mathfrak{a} \arrow[d, "\varrho"] \arrow[r, "\delta"] 
        & \mathfrak{g} \arrow[d, "\varphi"] \\
        C^\infty(M) \arrow[r, "\rho^* d"] 
        & \Gamma(E) \arrow[r, "\mathrm{ad}"] 
        & \Gamma(\mathbb{A}_E).
\end{tikzcd}
\end{equation}
\end{itemize} 
In the notation above, for $w\in \h$, $\mu^*w \in C^\infty(M)$ is given by $\mu^*w (x)=\mu(x) (w)$. We also use the notation $\mu_w=\mu^*w$. 
\hfill $\diamond$
\end{definition}

When the complex \eqref{eq:deltaseg} defines a short exact sequence 
\begin{equation}\label{eq: exact dgla}
0 \stackrel{}{\longrightarrow} \h \stackrel{\delta}{\longrightarrow} \mathfrak{a} \stackrel{\delta}{\longrightarrow} \g \stackrel{}{\longrightarrow} 0,  
\end{equation}
the ``exact'' dgla $(\ov\g, [\cdot,\cdot], \delta)$ (see $\S$ \ref{subsec:exact}) is equivalent to an exact Courant algebra  in the sense of \cite{Bursztyn} (\cite[Prop.~8.16]{GGGR}), in which case Def.~\ref{def:hamact} recovers the actions on exact Courant algebroids considered in that paper.

\subsection{Courant reduction}\label{subsection: Courant red}
Given a hamiltonian action $(\varphi, \varrho, \mu)$ of $(\g, \a, \h, \delta, \varpi)$ on $(E\to M, \SP{\cdot,\cdot}, \rho, \Cour{\cdot,\cdot}),$ for the reduction procedure we shall assume the following regularity conditions:
\begin{itemize}
\item[(r1)] $0\in \h^*$ is a regular value for $\mu\colon M\to \h^*;$ 
\item[(r2)] for each $x\in \mu^{-1}(0)$, the map 
\begin{equation}\label{eq:aK}
\mathfrak{a}\to E|_x, \quad a \longmapsto \varrho(a)|_x, 
\end{equation}
 is injective. 
\end{itemize}

Note that, by the $\g$-equivariance of $\mu: M \rightarrow \h^*$, for $x \in \mu^{-1}(0)$ we have $\underline{\varphi}(u)|_x \in T_x \mu^{-1}(0)$, and therefore the $\g$-action $\varphi: \g \rightarrow \Gamma(\bA_E)$ on $E$ restricts to a $\g$-action on $E|_{\mu^{-1}(0)}$.

Let $G$ be a connected Lie group integrating $\g$.

\begin{definition}\label{def:regact}
 The hamiltonian action $(\varphi, \varrho, \mu)$ is called {\em regular} if (r1) and (r2) hold, and the $\g$-action on $E|_{\mu^{-1}(0)}$ integrates to a $G$-action whose restriction to $\mu^{-1}(0)$ is free and proper.
 \hfill $\diamond$
\end{definition}



For $(\varphi, \varrho, \mu)$ regular, the images of \eqref{eq:aK} for each $x\in \mu^{-1}(0)$ define a vector subbundle $K$ of $E$ over the submanifold $\mu^{-1}(0)$, 
    \begin{equation}\label{eq:K}
    K\subseteq E|_{\mu^{-1}(0)}.
    \end{equation} 

Due to \eqref{eq:Kisot}, $K$ is isotropic. By the $\g$-equivariance of $\varrho$, 
and since $\g$ acts on $E$ by metric symmetries and $G$ is connected, both $K$ and $K^\perp$ are $G$-invariant. 
We then have a quotient vector bundle  $E_{red}=(K^\perp/K)/G$ over $M_{red}  = \mu^{-1}(0)/G$, fitting into the pullback diagram
\begin{equation}\label{diagram: quotient red}
\begin{tikzcd}
K^{\perp}/K \arrow[r, "\widehat p"] \arrow[d] 
& E_{\mathrm{red}} \arrow[d] \\
\mu^{-1}(0) \arrow{r}{p} 
& M_{\mathrm{red}},
\end{tikzcd}
\end{equation}
defined by the quotient maps; in particular, the pullback of sections $p^*: \Gamma( E_{\mathrm{red}}) \to \Gamma(K^\perp/K)$ identifies $ \Gamma( E_{\mathrm{red}})$ with the space of invariant sections $\Gamma(K^\perp/K)^G$.
For an element $e$ of $K^\perp$, we denote its image in $K^\perp/K$ by $[e]$.

Since the pseudo-riemannian metric $\SP{\cdot, \cdot}$ on $E$ is $G$-invariant, so is the induced metric on  $K^\perp/K$; hence it naturally descends to a pseudo-riemannian metric on $E_{red} = (K^{\perp}/K)/G$. As the next result recalls,  the Courant bracket $\Cour{\cdot, \cdot}$ and anchor $\rho$ can also be  reduced to $E_{red}$.

\begin{theorem}[Hamiltonian reduction of Courant algebroids]\label{thm:geometricCred} 
 Let $(\varphi, \varrho, \mu)$ be a regular hamiltonian action of $(\g, \a, \h, \delta, \varpi)$  on  $(E, \SP{\cdot,\cdot}, \rho, \Cour{\cdot,\cdot}).$  Then the pseudo-riemannian vector bundle $E_{red}\to M_{red}$  inherits a Courant algebroid structure, with anchor $\rho_{red}$ and bracket $\Cour{\cdot, \cdot}_{red}$ defined by 
\begin{equation}\label{eq: red courant str}
    \rho_{red}(\underline{e}) \coloneq Tp(\rho(e)), \quad  p^*\Cour{\underline{e}_1, \underline{e}_2}_{red} \coloneqq [ \,\Cour{ {e}_1, {e}_2}|_{\mu^{-1}(0)} \,],
\end{equation}
where $e\in K^\perp$ is such that $\widehat{p}([e])=\underline{e}$, and $e_1, e_2$ are sections of $E$ such that $e_i|_{\mu^{-1}(0)} = p^*\underline{e}_i$, $i=1,2$.
\end{theorem}

This result is a special instance of \cite[Thm.~6.11]{GGGR} (a proof similar to the ``exact'' case in   \cite[Theorem~3.3]{Bursztyn} also yields the result).

\subsection{Homological Courant reduction: main result}\label{subsec:main}
Starting with a hamiltonian action on a Courant algebroid, our main result builds a dg Courant-Dorfman algebra $(\cR, \cE)$ whose associated Courant-Dorfman algebra in cohomology (Prop.~\ref{prop: CD algebra in cohomology}) agrees with the reduced Courant algebroid of Theorem~\ref{thm:geometricCred}. 

For a vector space $V$, we use the notation $V_M:= V\times M$ for the trivial vector bundle.

Let $(\g, \a, \h, \delta, \varpi)$ be a dgla
and $(E\to M, \SP{\cdot,\cdot}, \rho, \Cour{\cdot,\cdot})$ a Courant algebroid.

Let
$$
\cR:= C^\infty(M)\otimes \wedge \g^* \otimes \wedge \h, 
$$
be the graded commutative algebra
where elements in $\g^*$ have degree $1$, and  elements in  $\h$ have degree $-1$. 
We have a natural identification
$$
\mathrm{Der}(\cR) = \cR \otimes_{C^\infty(M)}(\Gamma(TM \oplus \g_M[1] \oplus  \h_M^*[-1])),
$$
in which $u \in \g$ acts as a derivation of $\wedge \g^*$ via the contraction $\iota_u$, and similarly $w\in \h^*$ acts on $\wedge \h$ via $\iota_w$.

Consider the graded vector bundle $E [-1] \oplus \a_M^*[-2] \oplus \a_M$ over $M$, recalling that elements in $E$ have degree $1$, elements in $\a_M^*$ have degree 2, and elements in $\a_M$ have degree $0$, and consider the graded $\cR$-module
$$
\cE:= \cR\otimes_{C^\infty(M)}(\Gamma(E [-1] \oplus \a_M^*[-2] \oplus \a_M)). 
$$

Let $\dCE: \wedge^\bullet \g^* \rightarrow  \wedge^{\bullet +1}\g^*$ be the Chevalley-Eilenberg (CE) differential of $\g$ with respect to the trivial $\g$-module $\bR$, and for the $\g$-modules $\h$, $ \a$ and $\a^*$, let $\dCE^\h$, $ \dCE^{\a}$ and $ \dCE^{\a^*}$ be the corresponding CE differentials.

\begin{theorem}\label{thm:main}
    For a hamiltonian action $(\varphi,\varrho,\mu)$  of a dgla  $(\g, \a, \h, \delta, \varpi)$ on $(E\to M, \SP{\cdot,\cdot}, \rho, \Cour{\cdot,\cdot})$, the  pair $(\cR, \cE)$ carries a nondegenerate dg Courant-Dorfman algebra structure of degree $-1$ defined as follows:
    \begin{itemize}
\item the differential $\dR: \cR^\bullet \to \cR^{\bullet +1}$ is given on generators $f \in C^\infty(M)$, $ u^* \in \g^*$, and $w \in \h$ by
$$
\dR(f)\in C^\infty(M)\otimes \g^*, \qquad \dR(f)(u)=\Lie_{\underline{\varphi}(u)}f,     
$$
and
$$
\dR(u^*) =  \dCE(u^*) \quad \text{and} \quad \dR(w) = \mu^*(w) + \dCE ^\h(w),
$$
for $ u\in \g$, and $\underline{\varphi}$ as in \eqref{eq:Mact}; 
\item the differential $\dE: \cE^\bullet\to \cE^{\bullet+1}$ is determined by
$$
\dE(e)\in C^\infty(M)\otimes \a^* + \g^*\otimes \Gamma(E), \quad   \dE(e)(v) = - \SP{\varrho(v),e},  \;\; \dE(e)(u) =\varphi(u)e,
$$
and
$$
\dE(v^*) = d_{\scriptscriptstyle{CE}}^{\a^*}(v^*),   \quad \dE(v) = \varrho(v) + \iota_v\varpi + d^\a_{\scriptscriptstyle{CE}}(v),
$$
for $e \in \Gamma(E),$ $ u \in \g,$ $ v\in\a $ and $v^* \in \a^*$; 
\item the anchor  $\boldsymbol{\rho}_{\scriptscriptstyle{\cE}}:  \cE \rightarrow \mathrm{Der}(\cR)$ is determined by 
 $$
 \rhocE (e) = \rho(e), \quad \rhocE(v) = \iota_{\delta(v)} \quad \text{and} \quad \rhocE(v^*) =  \iota_{\delta^*(v^*)},
 $$ 
for $e \in \Gamma(E),$ $v\in \a$ and $v^*\in \a^*$; 

\item the bilinear map $\SP{\cdot, \cdot}_{\scrsE}: \cE \otimes_{\cR} \cE \rightarrow \cR$
is determined by
$$ 
\SP{e_1, e_2}_{\scrsE} = \SP{e_1, e_2}, \qquad \SP{v^\prime, v^*}_{\scrsE} = v^*(v^\prime) = - \SP{v^*,v^\prime}_{\scrsE},
$$
and
$$
 \SP{e, v^\prime}_{\scrsE} = 0 = \SP{e, v^*}_{\scrsE}, \qquad  \SP{v_1, v_2}_{\scrsE}=0, \;\;  \SP{v_1^*,v_2^*}_{\scrsE}=0,
$$
for all $e, e_1, e_2 \in \Gamma(E),$ $v_1, v_2, v^\prime\in \a$ and $v_1^*, v_2^*, v^* \in \a^*;$

\item the bracket $\Cour{\cdot,\cdot}_{\scriptscriptstyle{\cE}}: \cE \otimes_{\bR} \cE \rightarrow \cE$ is determined by 
$$
\Cour{e_1, e_2}_{\scrsE} \coloneqq  \Cour{e_1, e_2} \quad \text{and} \quad \Cour{e, v}_{\scrsE} =  \Cour{e, v^*}_{\scrsE} = \Cour{v, v^*}_{\scrsE} = \Cour{v_1,v_2}_{\scrsE} = \Cour{v_1^*, v_2^*}_{\scrsE}=0
$$
for $e, e_1, e_2 \in \Gamma(E),$ $v, v_1,v_2 \in \a$ and $v^*,v_1^*,v_2^* \in \a^*.$ 
\end{itemize}
     
Moreover, under the regularity assumptions of Theorem \ref{thm:geometricCred}, the Courant-Dorfman algebra $(H^0_\cR, H^1_\cE)$ is identified with the Courant-Dorfman algebra $(C^\infty(M_{red}), \Gamma(E_{red}))$ of the reduced Courant algebroid. 
\end{theorem}

\begin{remark}[Relation with the BRST algebra of quotients]
From the perspective of \cite{Carvalho}, the commutative dg algebra
$(\cR, \dR)$ is the BRST algebra providing the homological model for the quotient of $\mu^{-1}(0) \subset M$ by the $\g$-action $\underline{\varphi}$, i.e., it is the dg algebra whose degree zero cohomology recovers the algebra of $\g$-invariant functions on $\mu^{-1}(0).$ 
\hfill $\diamond$
\end{remark}

This theorem will be derived and proved using the graded symplectic approach to Courant reduction \cite{GGGR} and  the construction of a homological (BFV) model for the reduction of symplectic $\bN$-manifolds of degree $2$ (cf. \cite{AleZabzine, Carvalho}).

\subsection{Example: the exact case}\label{subsec:exact}

We will illustrate the  dg Courant-Dorfman from Theorem~\ref{thm:main} when both the dgla and Courant algebroid are ``exact''.

Let $\g$ be a Lie algebra, $\h$ a $\g$-module, and $\Upsilon: \g\times \g \to \h$ a bilinear map satisfying
\begin{equation}\label{eq:cocycle}
 u_1 \Upsilon(u_2, u_3) - u_2\Upsilon(u_1, u_3) - \Upsilon([u_1, u_2], u_3) - \Upsilon(u_2, [u_1, u_3]) + \Upsilon(u_1, [u_2, u_3]) =0,
\end{equation}
for all $u_1, u_2, u_3 \in \g$. Then the triple 
$$
(\g,\h,\Upsilon)
$$ 
defines a dgla $(\g,\a,\h,\delta,\varpi)$, where
\begin{itemize}
\item $\a=\h\oplus\g$, with $\g$-module structure
$$
u_1 \cdot (w,u) := (u_1 w + \Upsilon(u_1,u), [u_1,u]);
$$
\item $\varpi: \a\otimes \a\to \h$, 
$$
\varpi(w_1, u_1 ), (w_2, u_2 )) = u_2w_1 + u_1w_2 + \Upsilon(u_1, u_2) +  \Upsilon(u_2, u_1);
$$
\item the differential $\delta$ is given by the inclusion $\h\hookrightarrow \a$ and projection $\a\twoheadrightarrow \g$.
\end{itemize}

\begin{remark}
Condition \eqref{eq:cocycle} on $\Upsilon: \otimes^2 \g\to \h$ can be seen as a cocycle condition in {\em Leibniz cohomology} by regarding $\h$ as a Leibniz bimodule over $\g$  with respect to its left $\g$-action and trivial right $\g$-module structure ($w\cdot \g = 0$, for all $w\in \h$), see \cite{Mandal}. \hfill $\diamond$
\end{remark}

A dgla $(\g,\a,\h,\delta,\varpi)$ is called {\em exact} when the complex \eqref{eq:deltaseg} is an exact sequence (of $\g$-modules),
$$
0 \longrightarrow \h \longrightarrow \a \longrightarrow \g \longrightarrow 0.
$$
In this case, a choice of (vector-space) splitting $s: \g\to \a$ induces an isomorphism of the given dgla with the one defined by $(\g,\h,\Upsilon)$ as above, for
$$
\Upsilon(u_1,u_2) = u_1 s(u_2)-s([u_1,u_2]).
$$

When $\Upsilon=0$, the corresponding dgla corresponds to the {\em hemisemidirect product} of $\g$ and $\h$, see \cite[$\S$ 8.5]{GGGR}.

We now introduce a class of hamiltonian actions of exact dglas on exact Courant algebroids.

Let $(\g,\h,\Upsilon)$ be a triple as above, and let $\bT_\eta M$ denote an  exact Courant algebroid defined by a closed 3-form $\eta\in \Omega^3(M)$ (see Example~\ref{ex:exactCA}).

Consider a $\g$-action $\psi: \g \to \mathfrak{X}(M)$, a $\g$-equivariant map $\mu: M\to \h^*$, and let $\zeta: \g \to \Omega^1(M)$ satisfy
\begin{itemize}
\item[(a)] $i_{\psi(u)}\zeta(u) = \mu^*(\Upsilon(u,u))$,
\item[(b)] $d(\mu^*(\Upsilon(u_1, u_2))) + \zeta([u_1, u_2])
= \Lie_{\psi(u_1)} \zeta(u_2) - \iota_{\psi(u_2)}d\zeta(u_1) + \iota_{\psi(u_2)}\iota_{\psi(u_1)}\eta$,
\end{itemize}
for $u$, $u_1$, $u_2\in \g$.

\begin{proposition}\label{prop:hamact}
The data $(\psi, \mu, \zeta)$ determines a hamiltonian action of the exact dgla corresponding to $(\g,\h,\Upsilon)$ on the exact Courant algebroid $E=\bT_\eta M$ by
\begin{itemize}
\item $\varrho: \h\oplus \g\to \Gamma(E)$, 
$$
\varrho(w,u)=(\psi(u), d\mu_w + \zeta(u)),
$$
\item $\varphi: \g\to \Gamma(\bA_E)$, $\varphi(u)=\Cour{\varrho(u),\cdot}_\eta$.
\end{itemize}

\end{proposition}

This recovers \cite[Prop.~2.3]{Bursztyn1} when $\Upsilon =0$.

Note that when the dgla is a hemisemidirect product ($\Upsilon =0$) and the exact Courant algebroid is standard ($\eta=0$), we have a {\em canonical} hamiltonian action defined by $\zeta=0$. We refer to this case as the {\em standard hamiltonian action}.

Let us fix a hamiltonian action of the dgla determined by $(\g, \h, \Upsilon)$ on $\bT_\eta M$ defined by $(\psi, \mu, \zeta)$. 
Recall that, by Theorem~\ref{thm:main}, it gives rise to a dg Courant-Dorfman algebra $(\cR,\cE)$, where
$$
\cR= C^\infty(M)\otimes \wedge \g^* \otimes \wedge \h,
$$
with $\g^*$ in degree $1$ and $\h$ in degree $-1$, and
$$
\cE  =  \cR \otimes_{C^\infty(M)}(\Gamma(E[-1] \oplus \a_M^*[-2] \oplus \a_M)).
$$

Using the decompositions $\a=\h\oplus \g$ and $E=TM\oplus T^*M$, we have an isomorphism of $\cR$-modules
\begin{equation}\label{eq:kappa}
\kappa: \cE\to \bT_\cR[-1]=\Der(\cR)[-1]\oplus \Omega^1_\cR[-1]
\end{equation}
defined on generators by
$$
\kappa((Y,\beta), (\lambda,\xi), (w,u)) = ((Y + \iota_u +\iota_\lambda),(\beta - \widehat{w} + \widehat{\xi}))
$$
where $(Y,\beta) \in \Gamma(E)=\mathfrak{X}(M)\oplus \Omega^1(M)$, $(\lambda,\xi)\in \a^*=\h^*\oplus \g^*$, $(w,u)\in \a=\h\oplus \g$, $\iota_u$, $\iota_\lambda$ are the usual contraction derivations, and $\widehat{w}$, $\widehat{\xi}$ are defined by
$$
\bm{i}_{\iota_\lambda}\widehat{w} = \widehat{w}(\iota_\lambda)=\lambda(w), \qquad \bm{i}_{\iota_u}\widehat{\xi} = \widehat{\xi}(\iota_u)=\xi(u),
$$
and zero on the other types of generators. 


\begin{theorem}\label{thm:exact}
Let $(\cR,\cE)$ be the dg Courant-Dorfman algebra associated with the hamiltonian action on $\bT_\eta M$ determined by $(\psi, \mu, \zeta)$. Then the map $\kappa$ identifies $(\cR,\cE)$ with the $(F,H)$-twisted dg Courant-Dorfman algebra  on $\bT_\cR[-1]$, where $H=\eta$ and $F$ is determined by
\begin{equation}\label{eq:F0-components}
\left\{
\begin{aligned}
&F(Y,Z)
\in C^\infty(M)\otimes\g^*, \quad
F(Y,Z)(u)=
\bigl(
\iota_{\psi(u)}\eta-d\zeta(u)
\bigr)(Y,Z),
\\
&F(Y,\iota_u)
=
-\zeta(u)(Y),
\\
&F(\iota_{u_1},\iota_{u_2})
=
\Upsilon(u_1,u_2)+\Upsilon(u_2,u_1)
\in\h\subseteq\cR^{-1},
\\
&F(\iota_u,\iota_\lambda)
\in\g^*\subseteq\cR^1,
\quad F(\iota_u,\iota_\lambda)(x)
=
-\lambda\bigl(\Upsilon(x,u)\bigr),
\\
&F(Y,\iota_\lambda)
=0,\quad 
F(\iota_{\lambda_1},\iota_{\lambda_2})
=0,
\end{aligned}
\right.
\end{equation}
for $Y,Z\in\mathfrak X(M)$, $u,u_1,u_2,x\in\g$, and
$\lambda,\lambda_1,\lambda_2\in\h^*$.
\end{theorem}

\begin{proof}
We first check the pairing. On the $E[-1]$-summand, $\kappa$ is the
identity, so it preserves the pairing of $TM\oplus T^*M$.
Let $v=(w,u)\in\h\oplus\g$, $v^*=(\lambda,\xi)\in\h^*\oplus\g^*$, so that
$\kappa(v)=(\iota_u,-\widehat w)$, $\kappa(v^*)=(\iota_\lambda,\widehat\xi)$.
Since $|v|=0$ and $|v^*|=2$, the canonical degree $-1$ pairing on
$\bT_\cR[-1]$ gives
$$
\SP{\kappa(v),\kappa(v^*)}
=
\bm{i}_{\iota_u}(\widehat\xi)
-
\bm{i}_{\iota_\lambda}(-\widehat{w}) =
\xi(u)+\lambda(w)
=
v^*(v) = \SP{v,v^*}_{\scriptscriptstyle{\cE}}.
$$
All other pairings between the three types of generators vanish
on both sides, so
\[
\SP{\kappa(e_1),\kappa(e_2)} = \SP{e_1,e_2}_{\scriptscriptstyle{\cE}}.
\]
One can also directly check that $\kappa$ preserves anchors (and hence the corresponding derivations, since the pairing on $\cE$ is nondegenerate). By Prop.~\ref{prop:exactdgCD}, the Courant-Dorfman algebra structure on $\bT_\cR[-1]$ induced by the isomorphism  $\kappa$ in \eqref{eq:kappa} is twisted by forms $(F,H)$, as in Example \ref{ex:twists}.

The form $F\in \Omega_\cR^{2,1}$ satisfies
\[
D:=\kappa\circ \dE\circ\kappa^{-1},
\qquad
D(A,\vartheta)
=
\bigl(d_{\Der}A,d_\Omega\vartheta + (-1)^{\m{A}}\bm{i}_A\bm{F}\bigr),
\]
for $(A,\vartheta) \in \Der(\cR)[-1]\oplus \Omega^1_\cR[-1]$. With the sign convention for the shifted contraction, it follows that
\begin{equation}\label{eq:FAB}
 F(A_0,B_0) = (-1)^{B_0} \SP{\dE \kappa^{-1}(A,0), \kappa^{-1}(B,0)}_{\scriptscriptstyle{\cE}},
\end{equation}
for $A, B\in \Der(\cR)[-1]$, and $A_0$, $B_0 \in \Der(\cR)$ the corresponding unshifted elements. Applying the previous formula to the three types of derivation
generators we obtain the defining components of $F$ in
\eqref{eq:F0-components}.

Finally, we transport the Courant-Dorfman bracket through $\kappa$.
On two generators belonging to $E[-1]$, $\kappa$ is the identity, and
the original bracket is the $\eta$-twisted exact Courant bracket:
\[
\Cour{(Y,\beta),(Z,\gamma)}_\eta
=
\left(
[Y,Z],
\Lie_Y\gamma-\iota_Zd\beta+\iota_Z\iota_Y\eta
\right).
\]
Hence the transported bracket on $\bT_\cR[-1]$ differs from the standard Dorfman
bracket by the $3$-form $H=\eta$.

If one of the entries is a generator in $\mathfrak a$ or
$\mathfrak a^*$, both the bracket on $\cE$ and the
corresponding standard twisted bracket vanish. Consequently,
\[
\kappa\bigl(\Cour{e_1,e_2}_\cE\bigr)
=
\Cour{\kappa(e_1),\kappa(e_2)}_{\eta},
\]
so $H=\eta$.
\end{proof}

\begin{corollary}\label{cor:stdaction}
For a standard hamiltonian action on $\bT M$, the corresponding dg Courant-Dorfman algebra on $\bT_\cR[-1]$ is the standard one ($F=0$, $H=0$).
\end{corollary}


\section{(Bi-)Graded hamiltonian Poisson algebras}\label{sec:hamPA}

\subsection{The graded case}\label{subsec:graded} Let $A=\oplus_{i\in \bN} A_i$ be a graded commutative algebra  equipped with a graded Poisson bracket of degree $-2$, i.e., $\{A_i,A_j\}\subseteq A_{i+j-2}$ and
\begin{itemize}
           \item $\{a, b\} = -(-1)^{(|a| - 2)(|b|- 2)} \{b, a\};$
           \item $\{a, bc\} = \{a, b\}c + (-1)^{(|a|-2)|b|}b\{a, c\};$
           \item $(-1)^{(|a|-2)(|c| - 2)}\{a,\{b, c \}\} + (-1)^{(|b|-2)(|a| - 2)}\{b ,\{c, a \}\} + (-1)^{(|c|-2)(|b| - 2)}\{c,\{a, b \}\} = 0.$
\end{itemize}

Let  $\theta \in A_3$ be such that 
$$
\{\theta,\theta\}=0.
$$ 
We refer to the triple  $(A,\{\cdot,\cdot\}, \theta)$ as a {\em graded hamiltonian Poisson algebra} of degree -2. 

A central observation (following \cite[Thm.~4.5]{Roytenberg}) is that, when $A_0$ is regarded as a commutative algebra and $A_1$ as an $A_0$-module, the pair $(A_0,A_1)$ acquires the structure of a Courant-Dorfman algebra; the pairing, derivation and bracket are given, respectively, by
\begin{equation}\label{eq:CD}
\SP{a_1,a'_1}=\{a_1,a'_1\},\quad \partial a_0= \{\theta,a_0\} \quad \text{and} \quad [a_1,a'_1]=\{\{\theta,a_1\},a'_1\},
\end{equation}
for  $a_1,a'_1 \in A_1$ and $a_0 \in A_0$. So the $A_0$-valued symmetric bilinear pairing on $A_1$ is given by the Poisson bracket, while the derivation $\partial: A_0 \rightarrow A_1$ and the bracket $[\cdot, \cdot ] : A_1 \otimes A_1 \rightarrow A_0$ are derived brackets \cite{Yvette2} generated by $\theta$. 
We will extend  this result to the context of dg Courant-Dorfman algebras by introducing an additional grading on $A$ in the next subsection.

The following is a useful construction of Courant-Dorfman algebras via ``reduction''.
Given $(A, \{\cdot,\cdot\}, \theta)$, let $\cI \subseteq A$ be a homogeneous ideal that is {\em coisotropic}, i.e., $\{\cI,\cI\}\subseteq \cI$, and consider its normalizer,
$$
N(\cI)= \{ a\in A\,|\, \{a,\cI\}\subseteq \cI \}.
$$
Then $N(\cI)$ is a Poisson subalgebra of $A$, and $\cI\subseteq N(\cI)$ is a Poisson ideal. It follows that 
$N(\cI)/\cI$ inherits a graded Poisson bracket $\{\cdot,\cdot\}_{red}$ of degree -2.  If $\theta \in N(\cI)$ then it defines an element $\theta_{red} \in (N(\cI)/\cI)_3$ such that $\{\theta_{red},\theta_{red}\}_{red}=0$, and therefore induces a Courant-Dorfman algebra structure on the pair
$$ 
((N(\cI)/\cI)_0, (N(\cI)/\cI)_1).
$$

\subsection{The bigraded case}\label{subsec:bigraded} Let $\cA = \oplus_{i,k}\cA^{i,k}$, with $i\in \bN$ and $k\in \bZ,$ be a bigraded algebra with associative product satisfying
\begin{equation}\label{eq:bigcomm}
\cA^{i,k}\cA^{j,l}\subseteq \cA^{i+j,k+l}, \quad ab= (-1)^{kl} ba.
\end{equation}
We shall refer to the first grading on $\cA^{\bullet, \bullet}$ as the \emph{weight} and to the second as the \emph{function degree}. We keep the notation $| \cdot |$ for the function degree, i.e.,  for $a\in \cA^{i,k}$, $\m{a}=k$. 

A {\em Poisson bracket of bidegree $(-2, -2)$} on $\cA$ is a $\bR$-linear map $\{\cdot, \cdot\}: \cA \times \cA \rightarrow \cA$ satisfying 
\begin{equation}\label{eq:bideg}
\{ \cA^{i,k},\cA^{j,l} \}\subseteq \cA^{i+j-2,k+l-2},
\end{equation}
and that is a  graded Poisson bracket of degree $-2$ with respect to the function degree of $\cA^{\bullet, \bullet}$.

\begin{definition}\label{def:bigHPA}    
 We call $(\cA, \{\cdot, \cdot\})$ a {\em bigraded Poisson algebra of bidegree $(-2, -2)$}. 
 We will refer to an element
\begin{equation}\label{eq: charge grd hamil Poisson}
S = \Theta + Q \in \cA, \quad \text{with} \quad \Theta\in \cA^{3,3}, \quad Q\in \cA^{2,3},
\end{equation}
satisfying $\{S,S\}=0$ as a {\em charge}. The triple $(\cA, \{\cdot, \cdot\}, S)$ is then called a {\em bigraded hamiltonian Poisson algebra} (of bidegree $(-2, -2)$).
\hfill $\diamond$
\end{definition} 

A class of examples associated with dg manifolds is discussed in  $\S$ \ref{subsec:cotang}.

\begin{remark}[Relation with the graded case]\label{rem:graded} Given $(\cA, \{\cdot, \cdot\}, S)$ as in the previous definition, note that
$$
A= \oplus_{k\in \bN}\cA^{k,k} \subseteq \cA
$$
is a graded commutative algebra where the restriction of $\{\cdot,\cdot\}$ defines a graded Poisson bracket of degree $-2$, and therefore $(A, \{\cdot,\cdot\},\Theta)$ is a graded hamiltonian Poisson algebra of degree $-2$ (as in $\S$ \ref{subsec:graded}). On the other hand, any 
graded hamiltonian Poisson algebra of degree $-2$ can be identified with a bigraded hamiltonian Poisson algebra $(\cA, \{\cdot, \cdot\}, S)$ such that 
$\cA^{i,k}=\{0\}$ whenever $i\neq k$, i.e., 
$\cA= \oplus_k \cA^{k,k}$ (in this case, $Q \in \cA^{2, 3}$ vanishes and we are left with $\Theta \in \cA^{3, 3}$ satisfying $\{\Theta, \Theta\} = 0$).

\hfill $\diamond$
\end{remark}

By \eqref{eq:bideg}, the master equation $\{S, S\} = 0$ is equivalent to 
$$
\{Q, Q\} = 0, \quad \{Q, \Theta\} = 0, \;  \mbox{ and }\, \{\Theta, \Theta\} = 0.
$$
Note that the equations $\{Q, Q\} = 0$ and $\{\Theta, \Theta\} = 0$ are not trivial  because the Poisson bracket $\{\cdot, \cdot\}$ is of function degree $-2$ (even) and both $Q$ and $\Theta$ have function degree 3 (odd). 

Defining 
$$
d \coloneqq \{Q, \cdot\} \;\; \mbox{ and } \;\; \partial \coloneqq \{\Theta, \cdot\},
$$ 
we see that the operator $\{S,\cdot\}$ on $\cA^{\bullet, \bullet}$ decomposes according to the following diagram.

\begin{equation}\label{eq: diagram dg CD}
\begin{tikzcd}
\phantom{h} & \phantom{h} & \iddots & \iddots  & \iddots \\
{\scriptstyle{fdeg = \, 1}} & \cA^{0, 1} \arrow{ru}{\partial_0}  & \cA^{1, 1} \arrow{ru}{\partial_1} & \cA^{2, 1} \arrow{ru}{\partial_2}   & \iddots \\
{\scriptstyle{fdeg = \, 0}} & \cA^{0, 0} \arrow{ru}{\partial_0} \arrow{u}{d_0}  & \cA^{1, 0} \arrow{ru}{\partial_1}\arrow{u}{d_1} & \cA^{2, 0} \arrow{ru}{\partial_2} \arrow{u}{d_2} & \iddots \\
{\scriptstyle{fdeg = \, - 1}}& \cA^{0, -1} \arrow{u}{d_0} \arrow{ru}{\partial_0} & \cA^{1, -1} \arrow{ru}{\partial_1} \arrow{u}{d_1} & \cA^{2, -1} \arrow{u}{d_2}\arrow{ru}{\partial_2} & \iddots \\ 
\phantom{h} & \iddots \arrow{ru}{\partial_0} & \iddots \arrow{ru}{\partial_1} & \iddots \arrow{ru}{\partial_2} \\
\phantom{h} & {\scriptstyle{wght = \, 0}} & {\scriptstyle{wght = \, 1}} & {\scriptstyle{wght = \, 2}}
\end{tikzcd}
\end{equation}

Since $Q \in \cA^{2, 3}$ and $\Theta \in \cA^{3, 3}$ each satisfy the master equation, the Jacobi identity for $\{\cdot, \cdot\}$ implies that $d \coloneqq \{Q, \cdot \} $ and $\partial \coloneqq \{\Theta, \cdot \}$ are differentials on $\cA^{\bullet, \bullet}$. So, at each fixed weight level $k$, we have a vertical complex $(\cA^{k, \bullet}, d_k = \{Q, \cdot\}|_{\cA^{k, \bullet}})$ and, along the diagonals, we have complexes with differentials $\partial_k : \cA^{k, l} \rightarrow \cA^{k+1, \ell +1}$. Note, in particular, that $(\cA^{0, \bullet}, d_0)$ is a dg algebra and  $(\cA^{1, \bullet}, d_1)$ is a dg module over $(\cA^{0, \bullet}, d_0)$. 

The next result extends the construction in $\S$ \ref{subsec:graded} of usual Courant-Dorfman algebras from graded hamiltonian Poisson algebras of degree $-2$.

\begin{theorem}\label{thm: derived dg CD}
Let $(\cA, \{\cdot,\cdot\}, S)$ be a bigraded hamiltonian Poisson algebra of bidegree $(-2, -2)$. Then $(\cA^{0,\bullet},\cA^{1,\bullet})$ carries the structure of a  dg Courant-Dorfman algebra (of degree $-1$) with differentials $d_0 \coloneqq \{Q,\cdot\}|_{\cA^{0, \bullet}}, $ $d_1 \coloneqq \{Q,\cdot\}|_{\cA^{1, \bullet}}$ 
and pairing, derivation and bracket given, respectively, by
\begin{equation}\label{eq: symm prod, deriv, bracket - thm dgCD}
\SP{a^1_1,a^1_2}=\{a^1_1,a^1_2\}, \quad   \partial(a^0)= \{\Theta,a^0\} \quad \text{and}  \quad [a^1_1,a^1_2]= - (-1)^{|a^1_1|}\{\{\Theta,a^1_1\},a^1_2\},
\end{equation}
for $a^1_1, a^1_2 \in \cA^{1, \bullet}$ and $a^0 \in \cA^{0, \bullet}.$
\end{theorem}

It is evident from the construction that an isomoprhism of bigraded hamiltonian Poisson algebras restricts to an isomorphism of the corresponding dg Courant-Dorfman algebras.

\begin{proof}
We will verify that the differentials $d_0$, $d_1$ together with  the pairing, derivation and bracket in \eqref{eq: symm prod, deriv, bracket - thm dgCD} fulfill the conditions in Def.~\ref{def:dgDC} as a consequence of the master equation for $S$ along with the Jacobi and Leibniz identities for the Poisson bracket $\{\cdot, \cdot\}.$   

Since $\{\cdot, \cdot\}$  is of bidegree  $(-2, -2)$, the restriction $\SP{\cdot,\cdot}:=\{\cdot, \cdot\}|_{\cA^{1, \bullet}}$ defines a skew-symmetric pairing on $\cA^{1, \bullet}$  of degree $-2$ taking values in $\cA^{0, \bullet}$ (see \eqref{eq:bideg}).  Also, since $\cA$ is non-negatively graded with respect to the weight grading, $\{\cA^{1, \bullet}, \cA^{0, \bullet} \} = 0 $ and therefore, by the Leibniz rule for $\{\cdot, \cdot\}$, the pairing $\SP{\cdot, \cdot}$ is bilinear over $\cA^{0, \bullet}$.  Hence $\SP{\cdot, \cdot}$ is as in Def.~\ref{def:dgDC} (a).

The fact that the map $\partial_0 \coloneqq \{\Theta, \cdot\}|_{\cA^{0, \bullet}}: \cA^{0, \bullet} \rightarrow \cA^{1, \bullet + 1}$ is a derivation of degree 1 (as in Def.~\ref{def:dgDC} (b)) follows from the Leibniz rule of $\{\cdot,\cdot\}$.

The derived bracket given by  
$$
[a_1, a_2] \coloneqq - (-1)^{\m{a_1}} \{ \{\Theta, a_1 \}, a_2  \}, \qquad a_1, a_2 
\in \cA,$$
is of function degree $-1$ and satisfies 
$$
[a_1, [a_2,a_3]] = [[a_1, a_2], a_3] + (-1)^{(\m{a_1} - 1)(\m{a_2}- 1)}[ a_2, [a_1,a_3]]
$$
as a consequence of $\{\Theta, \Theta\} = 0$  \cite{Yvette - derived bracket, Voronov - Higher brackets}. By \eqref{eq:bideg}, this bracket preserves $\cA^{1,\bullet}$. It follows that  $[\cdot, \cdot  ] : \cA^{1, \bullet} \otimes_\bR \cA^{1, \bullet} \rightarrow \cA^{1, \bullet}$ is a Leibniz bracket of degree $-1$ on $\cA^{1,\bullet}$, as in Def.~\ref{def:dgDC} (c).

By the Jacobi identity for $\{\cdot,\cdot\}$, the pairing $\SP{\cdot, \cdot}$ on $\cA^{1, \bullet}$ and the differentials $d_0 \coloneqq \{Q,\cdot\}|_{\cA^{0, \bullet}}$ and $d_1 \coloneqq \{Q,\cdot\}|_{\cA^{1, \bullet}}$ satisfy 
$$d_0 \langle a^1_1, a^1_2 \rangle = \langle d_1 a^1_1, a^1_2 \rangle + (-1)^{|a^1_1|} \langle a^1_1, d_1 a^1_2 \rangle, 
$$
for $a^1_1, a^1_2 
\in \cA^{1, \bullet}$, as in Def.~\ref{def:dgDC} (ii.).

Using that $\{Q, \Theta\} = 0$, it results from the Jacobi identity for $\{\cdot, \cdot\}$ that the differential $d_1 : \cA^{1, \bullet} \rightarrow \cA^{1, \bullet} $ and the Leibniz bracket $[\cdot, \cdot  ] : \cA^{1, \bullet} \otimes \cA^{1, \bullet} \rightarrow \cA^{1, \bullet}$ satisfy 
$$
d_1 [a^1_1, a^1_2] = [d_1 a^1_1 ,  a^1_2] + (-1)^{|a^1_1| - 1 } [a^1_1, d_1 a^1_2], \qquad a^1_1, a^1_2 
\in \cA^{1, \bullet}.
$$
Likewise, we obtain that 
$
d_1 \circ \partial_0 + \partial_0 \circ d_0 = 0. 
$
Thus, conditions (i.), (ii.), and (iii.) of Def.~\ref{def:dgDC} hold.

To conclude, one can show directly from  the Jacobi identity for the Poisson bracket $\{\cdot, \cdot\}$ and  $\{\Theta, \Theta\} = 0$  that conditions (1) to (5) in Def.~\ref{def:dgDC} are satisfied.
\end{proof}

\begin{remark}[Degree $-m$]
The construction in Theorem~\ref{thm: derived dg CD} admits an
extension to dg Courant-Dorfman algebras of degree $-m$: one replaces the graded commutativity property \eqref{eq:bigcomm} by
\[
ab =(-1)^{kl+(m+1)ij} ba, \qquad a\in \cA^{i,k}, b\in \cA^{j,l},
\]
and considers a Poisson bracket of bidegree $(-2,-2m)$ (defined by the previous commutation factor, i.e., $\{a,b\} = -(-1)^{kl+(m+1)ij}\{b,a\}$, etc.).
 One then takes $Q\in\cA^{2,\,2m+1}$, 
$\Theta\in\cA^{3,\,3m}$,
satisfying
$\{Q,Q\}=0$, $\{Q,\Theta\}=0$, and
$\{\Theta,\Theta\}=0$.
Then the formulas
$d=\{Q,\cdot\}$, $\SP{a_1,a_2}=\{a_1,a_2\}$, 
$\partial a=\{\Theta,a\}$, and
$$
[a_1,a_2] = (-1)^{m(|a_1|-1)}
\{\{\Theta,a_1\},a_2\}
$$
define on $(\cA^{0,\bullet},\cA^{1,\bullet})$
a dg Courant-Dorfman algebra of degree $-m$. 
\hfill $\diamond$
\end{remark}

\begin{remark}[Other generalizations]\
\begin{itemize}
\item[(a)] Theorem~\ref{thm: derived dg CD} holds sheaf theoretically, i.e., it yields a sheaf of dg Courant-algebras algebras from a sheaf of bigraded hamiltonian Poisson algebras of bidegree $(-2,2)$.
\item[(b)] By considering additional terms in the charge $S$, one naturally obtains {\em homotopical versions} of Courant-Dorfman algebras, which are enhancements of Leibniz infinity algebras \cite{Poncin,Uchino}.
\end{itemize}
\hfill $\diamond$
\end{remark}

Given $(\cA, \{\cdot,\cdot\}, S=Q+\Theta),$
denote by $H_Q^{\bullet,\bullet}(\cA)$ the vertical cohomology of \eqref{eq: diagram dg CD}, i.e., 
$H^{k,l}_Q(\cA)$ is the $l$-th cohomology group of the complex $(\cA^{k,\bullet},d_k)$. 
A direct consequence of the previous theorem and Prop.~\ref{prop: CD algebra in cohomology} is that
the pair 
\begin{equation*}\label{eq:cohCD}
(H^{0,0}_Q(\cA),H^{1,1}(\cA))
\end{equation*}
acquires the structure of a Courant-Dorfman algebra. This fact can also be understood in light of the following observations. Since $d=\{Q,\cdot\}$ is a Poisson derivation, the Poisson bracket on $\cA$ descends to a Poisson bracket on $H_Q(\cA)$ via 
$$
\{[a], [b]\} = [\{a,b\}],
$$
where $[a]$, $[b]$ denote cohomology classes of $d$-closed elements $a, b \in \cA$.

\begin{proposition}\label{prop:Qcoh}
The triple $(H_Q(\cA), \{\cdot,\cdot\},[\Theta])$   is a bigraded hamiltonian Poisson algebra of bidegree $(-2,-2)$. 
\end{proposition}

It follows (see Remark~\ref{rem:graded}) that $\oplus_k H_Q^{k,k}(\cA)$ inherits the structure of a graded hamiltonian Poisson algebra of degree $-2$, which induces (as explained in $\S$ \ref{subsec:graded}) a Courant-Dorfman algebra on $(H^{0,0}_Q(\cA),H^{1,1}_Q(\cA))$, which coincides with the one described above.

Our next step toward proving Theorem~\ref{thm:main} is to construct a bigraded hamiltonian Poisson algebra from the data of a hamiltonian action on a Courant algebroid. This will be achieved via a BFV-type construction. 


\section{Graded symplectic viewpoint}\label{sec:graded}

In this section, we briefly recall how Courant algebroids and their hamiltonian reductions can be interpreted in graded geometric terms, see \cite{RoytenbergThesis,Roytenberg} and \cite{GGGR}.

\subsection{Symplectic $\bN Q$-manifolds of degree 2}
 An $\mathbb{N}$-manifold of degree 2  of dimension $n_0|n_1|n_2$, denoted by $\cM$, consists of an $n_0$-dimensional smooth manifold  $M$ endowed with a (non-negatively) graded commutative sheaf of algebras $C_\cM$  satisfying the following local property: any $x \in M$ admits an open neighborhood $U \subseteq M$ such that  
    \begin{equation}\label{eq: local triv sheaf deg 2 mnf}
    C_\cM|_U \cong C^\infty_U \otimes \wedge^\bullet \bR^{n_1} \otimes S^\bullet \bR^{n_2},
    \end{equation}
where the elements of $\bR^{n_i}$ have degree $i$. 
The local condition implies that, over sufficiently small open subsets $U\subseteq M$,
there exist local sections 
\begin{equation}\label{eq:localcoord}
x^1,\ldots,x^{n_0}, e^1,\ldots, e^{n_1}, p^1,\ldots,p^{n_2},
\end{equation}
where $\{x^i\}$ are usual local coordinates on $U$, the sections $\{e^\mu\}$ have degree 1, and $\{p^I\}$ have degree 2, in such a way that each homogeneous section of degree $q$ is a degree $q$ polynomial in the variables $e^\mu$ and $p^I$ over $C^\infty(U)$. 

We often refer to sections of $C_\cM$ as ``functions'' on $\cM$.
We denote by $(C_\cM)_q$ the subsheaf of degree $q$ sections of $C_\cM$. It follows from \eqref{eq: local triv sheaf deg 2 mnf} that $(C_\cM)_0$ is identified with $C^\infty_M$, and that, for each $q=1,2,\ldots$, $(C_\cM)_q$ is a locally free sheaf of $C^\infty_M$-modules, and therefore is the sheaf of sections of a vector bundle. We use the notation $C(\cM)=C_\cM(M)$ and $C_q(\cM)=(C_\cM)_q(M)$ for the corresponding spaces of global sections.

A {\em symplectic structure} on $\cM$ is a Poisson bracket $\{\cdot,\cdot\}$ on $C_\cM$ of degree $-2$ that is {\em nondegenerate}. In local coordinates \eqref{eq:localcoord}, nondegeneracy means that the matrices $(\{p^I,x^j\})$ and $(\{e^\mu,e^\nu\})$ are pointwise invertible; in particular, this implies that $n_0=n_2$.

Consider a degree 2 symplectic $\bN$-manifold $(\cM, \{\cdot,\cdot\})$ together with a function $\theta\in C_3(\cM)$ satisfying $\{\theta,\theta\}=0$. Recall that this last condition implies that the derivation $Q:=\{\theta,\cdot\}$ of $C_\cM$ satisfies $Q^2=0$, so it is a ``homological'' hamiltonian vector field on $\cM$. We refer to the triple $(\cM, \{\cdot,\cdot\}, \theta)$ as a {\em symplectic $\bN Q$-manifold of degree 2}.


\subsection{Courant algebroids}\label{subsection: grd view courant alg}
Courant algebroids admit an alternative formulation as symplectic $\bN Q$-manifolds of degree 2, 
\begin{equation}\label{eq:RS}
\;\quad \mbox{ Courant algebroids} \quad \rightleftharpoons \quad  \mbox{Symplectic $\bN Q$-manifold of degree 2}.
\end{equation}
\smallskip
In this equivalence,  pseudo-Riemannian vector bundles $(E, \SP{\cdot,\cdot})$ correspond to degree 2 symplectic $\bN$-manifolds $(\cM, \{\cdot,\cdot\})$ \cite[Theorem~3.3]{Roytenberg} (see also \cite[$\S$~3.2]{GGGR}), while Courant structures $\rho$, $\Cour{\cdot,\cdot}$ on $(E, \SP{\cdot,\cdot})$ correspond to self-commuting functions $\theta\in C_3(\cM)$ \cite[Theorem~4.5]{Roytenberg}. We thus refer to such $\theta$ as a {\em Courant charge}.

One direction of \eqref{eq:RS} follows from the construction in $\S$~\ref{sec:hamPA}. 
Let $(\mathcal{M}, \{\cdot,\cdot\}, \theta)$ be a symplectic $\mathbb{N}Q$-manifold of degree 2, and let $E\to M$ be the vector bundle determined by $C_1(\cM)=\Gamma(E)$. 
Let 
\begin{equation}\label{eq:gralg}
\Fun := \oplus_{i=0}^\infty C_i(\cM) \subset C(\cM),
\end{equation}
which is a graded commutative algebra. 
The Poisson bracket on $C(\cM)$ restricts to $\Fun$, and $(\Fun, \{\cdot,\cdot\}, \theta)$ is a graded hamiltonian Poisson algebra of degree -2. Following $\S$~\ref{sec:hamPA}, this determines a nondegenerate Courant-Dorfman algebra $(C^\infty(M), \Gamma(E))$, and hence a Courant algebroid structure on $E$, via (cf. \ref{eq:CD})
\begin{equation}\label{eq: Courant str RS corresp}
\SP{e_1,e_2}=\{e_1,e_2\}, \quad \Lie_{\rho(e)}f = \{\{\theta, e\}, f\} \quad \text{and} \quad 
\Cour{e_1,e_2}= \{\{\theta, e_1\},e_2\},
\end{equation}
for $e, e_1,e_2 \in \Gamma(E)$ and $f\in C^\infty(M)$. 


The converse direction follows from the fact that every Courant algebroid arises as in the previous construction, from a canonically associated degree $2$ symplectic $\bN Q$-manifold. Given a Courant algebroid $(E, \SP{\cdot,\cdot}, \rho, \Cour{\cdot,\cdot})$, the corresponding graded manifold $\cM$ is characterized by
\[
C_1(\cM) =\Gamma(E), \qquad C_2(\cM)=\Gamma(\bA_E),
\]
where $\bA_E$ is the Atiyah algebroid of the pseudo-Riemannian vector bundle $(E, \SP{\cdot,\cdot})$, with Poisson bracket determined by the relations
$$
\{e_1,e_2\} =\SP{e_1,e_2}, \quad\;
\{\widetilde{e},f\} =\Lie_Xf, \quad\;  
\{\widetilde{e},e\} =D(e), \quad \;
\{\widetilde{e}_1, \widetilde{e}_2\} =([X_1,X_2],[D_1,D_2]), 
$$ 
for $f\in C^\infty(M)$, $e, e_i \in \Gamma(E)$, $\widetilde{e}=(X,D), \widetilde{e}_i=(X_i,D_i) \in \Gamma(\mathbb{A}_E)$, $i=1,2$. The Courant charge $\theta$ can be written in local coordinates \eqref{eq:localcoord} as
\begin{equation}\label{eq: coordinate for theta eq from a courant structure}
     \theta = a_{\mu I} (x) e^\mu p^I  + \phi_{\mu \nu \eta}(x)e^\mu e^\nu e^\eta,
\end{equation}
with the coefficients $a_{\mu I}$, $\phi_{\mu \nu \xi} \in C^{\infty}(M)$ uniquely determined by $\rho$ and $\Cour{\cdot,\cdot}$ by the last two equations in \eqref{eq: Courant str RS corresp}, see \cite[Thm.~4.5]{Roytenberg}.

\subsection{Hamiltonian actions}\label{subsection: grd view hamilt actions}
 We now recall the framework of graded hamiltonian actions from \cite[$\S$~8]{GGGR}.
Let $(\cM, \{\cdot,\cdot\}, \theta)$ be a symplectic $\bN Q$-manifold of degree 2, and let $(\overline{\g}, [\cdot,\cdot], \delta)$ be a dgla of the form $\ov\g=\h[2]\oplus \a[1]\oplus \g$.

Following the terminology of symplectic geometry \cite[$\S$~8.2]{GGGR}, we say that
$\cM$ is a {\em hamiltonian $\ov\g$-space} if it is equipped with a graded map
$$
J^\sharp: \ov\g[-2] \to \Fun\subset C(\cM)
$$
that preserves brackets and satisfies 
\begin{equation}\label{eq:diffpres}
J^\sharp(\delta(u)) = \{\theta, J^\sharp(u)\},
\end{equation}
for all $u\in \ov\g$.
Equivalently, 
\begin{equation*}\label{eq:dglamap}
J^ \sharp: \ov\g \to \Fun[2] \subset C(\cM)[2],
\end{equation*}
is a morphism of dglas from 
$(\overline{\g}, [\cdot,\cdot], \delta)$ to $(\Fun[2], \{\cdot,\cdot\},\{\theta,\cdot\})$. 

The graded map $J^\sharp$ is equivalent to a morphism of $\bN$-manifolds of degree 2,
\begin{equation}\label{eq:mommap}
J: \cM \to \ov\g^*[2],
\end{equation}
where $\ov\g^*[2]$ is the $\bN$-manifold with body $\h^*$ and sheaf of functions $C^\infty_{\h^*}\otimes \wedge^\bullet \a \otimes S^\bullet \g$, in which elements of $\a$ have degree 1 and elements of $\g$ have degree 2. Keeping the standard terminology, we say that the map $\ov \g \to \mathfrak{X}(\cM)=\mathrm{Der}(C(\cM))$, $u \mapsto \{J^\sharp u, \cdot\}$, is a {\em hamiltonian $\ov\g$-action} on $\cM$, with $J$ its {\em moment map}. With a slight abuse of notation, we also refer to $J^\sharp$ as a moment map. 

When $\delta=0$ and $\theta=0$, we recover the notion of a hamiltonian action of a graded Lie algebra $\ov\g$ on a degree 2 symplectic $\bN$-manifold $(\cM, \{\cdot,\cdot\})$.


Let $(E, \SP{\cdot,\cdot}, \rho, \Cour{\cdot,\cdot})$ be the Courant algebroid corresponding to $(\cM, \{\cdot,\cdot\}, \theta)$, and let $(\g, \a, \h, \delta, \varpi)$ correspond to $(\ov\g, [\cdot,\cdot], \delta)$, as in $\S$~\ref{subsection: hmlt actions Courant alg}. A graded map
$J^\sharp: \ov\g[-2] \to C(\cM)$
is specified by its components in degrees $0$, $1$, and $2$, namely by maps 
$$
\mu^*: \h \to C^\infty(M), \;\quad \varrho: \a \to \Gamma(E), \;\quad \varphi: \g \to \Gamma(\bA_E).
$$
The condition that $J^\sharp$ be a morphism of dglas \eqref{eq:dglamap} amounts to the requirement that $(\varphi,\varrho,\mu)$ defines a hamiltonian action on $(E, \SP{\cdot,\cdot}, \rho, \Cour{\cdot,\cdot})$ in the sense of Def.~\ref{def:hamact}, see \cite[$\S$~8.3]{GGGR}. In other words, the equivalence between Courant algebroids and symplectic $\bN Q$-manifolds of degree $2$ extends to an equivalence between hamiltonian actions in each setting,
\smallskip
$$
\begin{array}{c}
\text{Hamiltonian $\ov\g$-action on $(\cM,\{\cdot,\cdot\},\theta)$ with} \\
\text{moment map $J^\sharp: \ov\g[-2] \to C(\cM)$}
\end{array}\quad\;
\rightleftharpoons\quad\;
\begin{array}{c}
\text{Hamiltonian action $(\varphi,\varrho,\mu)$ of } (\g,\a,\h,\delta,\varpi) \\
\text{on }(E,\SP{\cdot,\cdot},\rho,\Cour{\cdot,\cdot})
\end{array}
$$
\smallskip

Moreover, with respect to this equivalence, the hamiltonian reduction of Courant algebroids (recalled in $\S$\ref{subsection: Courant red}) agrees with the  symplectic (Marsden-Weinstein) reduction in the realm of graded symplectic manifolds \cite[$\S$~8.4]{GGGR}. We will revisit this fact from the perspective of $\S$ \ref{sec:hamPA}.

Given a hamiltonian $\ov \g$-space with moment map $J^\sharp: \ov\g[-2]\to \Fun$, we let 
\begin{equation}\label{eq:coisoideal}
\cI=\langle J^\sharp(\ov \g[-2]) \rangle \subseteq \Fun
\end{equation}
be the homogeneous ideal generated by the image of $J^\sharp$. The fact that $J^\sharp$ preserves brackets implies that $\cI$ is coisotropic, i.e., $\{\cI,\cI\}\subseteq \cI$. Considering the normalizer of $\cI$ in $\Fun$,
$$
N(\cI) = \{ f\in \Fun \, |\,  \{f, \cI\} \in \cI \},
$$
it follows that $N(\cI)/\cI$ is a graded Poisson algebra of degree -2. By the compatibility of $J^\sharp$ with differentials \eqref{eq:diffpres}, we have that $\theta\in N(\cI)$, so $\theta$ defines an element $\theta_{red} \in (N(\cI)/\cI)_3$ satisfying $\{\theta_{red},\theta_{red}\}=0$. In this way $N(\cI)/\cI$ becomes a graded hamiltonian Poisson algebra of degree -2, and therefore induces a Courant-Dorfman algebra structure on the pair 
$$
((N(\cI)/\cI)_0, (N(\cI)/\cI)_1).
$$

Let us now assume the regularity conditions (r1), (r2) from $\S$\ref{subsection: Courant red}. From the graded-geometric viewpoint, these conditions mean that $0$ is a regular value for the moment map $J$ \eqref{eq:mommap}, so that  $\cI$ is the vanishing ideal of a graded submanifold $J^{-1}(0)\subset \cM$ \cite[$\S$ 2.4]{GGGR}. In this case, in terms of the hamiltonian action on the corresponding Courant algebroid, we have
\begin{equation}\label{eq:algreduc}
(N(\cI)/\cI)_0 = C^\infty(\mu^{-1}(0))^\g, \quad \quad \quad (N(\cI)/\cI)_1 = \Gamma(K^\perp/K)^\g,
\end{equation}
where $\Gamma(K^\perp/K)^\g$ and $C^\infty(\mu^{-1}(0))^\g$ denote, respectively, the spaces of $\g$-invariant sections of $K^\perp/K$ and $\g$-invariant functions on $\mu^{-1}(0)$ (see \cite[Prop.~4.15, Lem.~8.6]{GGGR}). These identifications endow the pair
$$
(C^\infty(\mu^{-1}(0))^\g, \Gamma(K^\perp/K)^\g)
$$ 
with the structure of a Courant-Dorfman algebra.
When the hamiltonian action  is regular (as in Def.~\ref{def:regact}), this Courant-Dorfman algebra is naturally identified with the one defined by the reduced Courant algebroid $E_{red} \rightarrow M_{red}$ of Theorem \ref{thm:geometricCred}. 

Denoting by $\cM_{red}$ the symplectic $\bN Q$-manifold of degree 2 corresponding to $E_{red} \rightarrow M_{red}$, we note that, by construction, its associated graded hamiltonian Poisson algebra is identified with $N(\cI)/\cI$. On the other hand, under the same regularity assumptions, $N(\cI)/\cI$
models the graded Poisson algebra of functions on the symplectic reduction of $\cM$ with respect to $J$ at level $0$, thereby showing that hamiltonian Courant reduction is the classical geometric counterpart 
of graded symplectic reduction.

\section{BFV in degree 2}\label{sec:BFV}

In this section, we present a BFV model for hamiltonian reduction of symplectic $\bN Q$-manifolds of degree 2, following the classical (non-graded) BFV construction \cite{Kostant, Stasheff}; see \cite{Carvalho} for the case of degree $1$.

\subsection{BFV algebra}\label{subsec:BFValg}

Consider a graded Lie algebra of the form $\ov\g = \h[2] \oplus \a[1] \oplus \g$, and let $(\cM, \{\cdot, \cdot\})$ be a symplectic $\bN$-manifold of degree $2$. Consider the graded commutative algebra
\begin{equation}\label{eq: bfv alg def}
\cA_{\scriptscriptstyle{BFV}} = \Fun\otimes S({\ov\g}^*[-1]) \otimes S(\ov\g[-1]),
\end{equation}
where the symmetric algebras have graded generators with the following degrees: 
\begin{center}
\begin{tabular}{c|c|c|c|c|c}
    $\h$ & $\a$ & $\g$ & $\g^*$ & $\a^*$ & $\h^*$ \\
     \hline
     -1 & 0 & 1 & 1 & 2 & 3
\end{tabular}
\hspace{0.1cm}.
\end{center}
We may regard  $\cA_{\scriptscriptstyle{BFV}}$ as the algebra of functions on the $\bZ$-graded manifold $\cM \times T^*[2]{\ov\g}[1]$, see $\S$ \ref{sec:interpret}.

We denote by $(\cA_{\scriptscriptstyle{BFV}})_{\ell} \subseteq \cA_{\scriptscriptstyle{BFV}}$ the subset of functions of degree $\ell$, and  let 
$$
A^{p,q}=\Fun\otimes S^p({\ov\g}^*[-1])\otimes S^q(\ov\g[-1]).
$$ 
We will say that $a^{p, q} \in A^{p, q}$ has {\em ghost number} $p,$ {\em antighost number} $q$ and {\em total ghost number} $p-q$  (cf. \cite{AleZabzine, Carvalho}).  We let $\cA^k_{\scriptscriptstyle{BFV}} \coloneqq \oplus_{k=p-q}A^{p,q}$, and denote by
$$
\cA^{k,\ell}_{\scriptscriptstyle{BFV}}=\cA^k_{\scriptscriptstyle{BFV}}\cap (\cA_{\scriptscriptstyle{BFV}})_\ell \subset \cA_{\scriptscriptstyle{BFV}}
$$ 
the subset of elements of total ghost number  $k$ and function degree $\ell$. We will often use the shorthand notations $tgh$ and $fdeg$ for the total ghost number and function degree, respectively.

We can further endow the graded algebra $\cA_{\scriptscriptstyle{BFV}}$ with a Poisson bracket of function degree $-2$, referred to as the {\em BFV bracket}  \cite{Stasheff}; this is defined by the Poisson bracket on $\Fun$ and the duality between ${\ov\g}^*[-1]$ and ${\ov\g}[-1]$, with the convention
$$
\{v,v^*  \}_{\scriptscriptstyle{BFV}}=v^*(v), \qquad \;\; v\in \a, \;\;v^*\in \a^*[-2],
$$
and extended to the whole algebra $\cA_{\scriptscriptstyle{BFV}}$ through the Leibniz rule. With respect to the ghost-antighost degrees, the BFV bracket satisfies
\begin{equation}
\{A^{i, j}, A^{k, \ell}\}_{\scriptscriptstyle{BFV}} \subset A^{i+k, j+ \ell} \oplus A^{i+k - 1, j + \ell - 1},
\end{equation}
so it is of degree zero with respect to the total ghost number.
If there is no risk of confusion, we denote $\{\cdot,\cdot\}_{\scriptscriptstyle{BFV}}$ simply by $\{\cdot,\cdot\}$.

In $\S$ \ref{subsec:CD-BFV} below we will replace the total ghost number $tgh$ by a new degree $wght:= fdeg - tgh$ with respect to which the BFV bracket has bidegree $(-2,-2)$, thereby fitting into the framework of $\S$ \ref{subsec:bigraded}.

\subsection{The BRST differential and cohomology}\label{subsec:BRSTcoh}
Suppose now that $(\cM, \{\cdot,\cdot\})$ is a hamiltonian $\ov\g$-space with moment  map
$$
J^\sharp: \ov\g[-2] \rightarrow \Fun.
$$ 

We encode these data in the so-called {\em BRST charge} $Q \in \cA^{1,3}_{\scriptscriptstyle{BFV}}$, 
where $Q =  Q^{1, 0} + Q^{2, 1}$, with
$$
Q^{1, 0} \in  \Fun\otimes S^1({\ov\g}^*[-1]) \quad \text{and} \quad Q^{2,1} \in S^2({\ov\g}^*[-1]) \otimes S^1({\ov\g}[-1])  
$$ 
completely determined by the following conditions:
\begin{equation}\label{eq: def conditions for BRST charge}
\{Q^{1, 0}, u\} = J^\sharp(u), \;\;\; \forall \, u\in \ov\g, \quad \text{and} \quad \{Q^{2, 1 }, \cdot \,  \}|_{\scriptscriptstyle{S^\bullet({\ov\g}^*[-1])}} = \dCE^{\ov\g}.
\end{equation}
In the first equality, we regard $u$ as an element of $S^1(\ov\g[-1])$ on the left-hand side and as an element of $\ov\g[-2]$  on the right-hand side; in the second equality, $\dCE^{\ov\g} : S^\bullet({\ov\g}^*[-1]) \rightarrow S^{\bullet + 1}({\ov\g}^*[-1]) $ denotes the Chevalley-Eilenberg differential for the graded Lie algebra $\ov\g$. 

A direct computation shows that 
\begin{align*}
    \{Q^{1,0}, Q^{1, 0}\} + 2 \{Q^{1, 0}, Q^{2, 1}\} & = 0,
\end{align*}
while  the Jacobi identity for the Lie bracket on $\ov\g,$ ensures that $\{Q^{2, 1}, Q^{2, 1}\} = 0.$
It follows that the BRST charge $Q \in \cA_{BFV}^{1,3}$ satisfies
\begin{equation}
    \{Q, Q\} = \{Q^{1,0}, Q^{1, 0}\} + 2 \{Q^{1, 0}, Q^{2, 1}\} + \{Q^{2, 1}, Q^{2, 1}\} = 0.
    \end{equation}

The {\em BRST differential} is defined  by $d_{\scriptscriptstyle{BRST}} \coloneqq \{Q, \cdot\}.$ Since $tgh(Q)= 1$ and $tgh(\{\cdot,\cdot\})= 0,$ we obtain a complex
$$
\begin{tikzcd}[column sep=large] 
\cdots \arrow{r}{d_{\scriptscriptstyle{BRST}}} & \cA^{ k}_{\scriptscriptstyle{BFV}}\arrow{r}{d_{\scriptscriptstyle{BRST}}} &  \cA^{k+1}_{\scriptscriptstyle{BFV}}  \arrow{r}{d_{\scriptscriptstyle{BRST}}} & \cdots,
\end{tikzcd}
$$
so $(\cA^\bullet_{\scriptscriptstyle{BFV}}, \{\cdot,\cdot\}, d_{\scriptscriptstyle{BRST}})$ is a differential graded Poisson algebra (with grading given by $tgh$).

\begin{remark}\label{rem: coord brst charge}
    For computations, it is useful to write the BRST charge $Q \in \cA^{1,3}_{BFV}$ in terms of a basis for $\ov\g=\h[2]\oplus \a[1]\oplus \g$. Let us fix $\{w^i\},\{v^j\}, \{u^k\}$ bases for $\h, \a$ and $\g$, and denote the dual bases by $\{w^*_i\},\{v^*_j\}, \{u^*_k\}$. 
    We obtain from \eqref{eq: def conditions for BRST charge} that 
    \begin{align}
        Q^{1, 0} & = J^\sharp_0(w^i)w^*_i - J^\sharp_1(v^j)v^*_j + J^\sharp_2(u^k)u^*_k, \label{eq: Q10 comp BRST}\\
        Q^{2, 1} & = - \frac{1}{2}c_k^{ij} u^*_iu^*_j u^k
        -d^{mn}_p u_m^* v_n^* v^p 
        - e^{ij}_k u_i^* w_j^* w^k
        - \frac{1}{2} f^{mn}_p v_m^* v_n^* w^p, \label{eq: Q21 comp BRST} 
    \end{align}
    where the coefficients in \eqref{eq: Q10 comp BRST} are the generators of the coisotropic ideal $\cI \subset C(\cM)$ in \eqref{eq:coisoideal}, and those in \eqref{eq: Q21 comp BRST} are the structure constants of the graded Lie algebra $\ov\g$ with respect to the given basis.
    \hfill $\diamond$
\end{remark}

We refer to the triple
$$
(\cA_{\scriptscriptstyle{BFV}}, \{\cdot,\cdot\}, Q)
$$ 
as the {\em BFV algebra}. This is the degree $2$ analogue of the BFV algebra constructed in \cite{Carvalho} in the context of hamiltonian reduction in degree $1$. 

The following properties of the cohomology $H^\bullet_{{\scriptscriptstyle{BRST}}}$ of $(\cA^\bullet_{\scriptscriptstyle{BFV}}, d_{\scriptscriptstyle{BRST}})$ are parallel to those in \cite[Lemma~5.3]{Carvalho}, and verified with similar proofs:
\begin{itemize}
    \item The cohomology $H^\bullet_{{\scriptscriptstyle{BRST}}}$ (which is computed with respect to the total ghost number) is graded by the function degree, i.e., for $p \in \bZ,$ 
\begin{equation}
 H^{p}_{{\scriptscriptstyle{BRST}}} \coloneqq H^{p, \bullet}_{{\scriptscriptstyle{BRST}}} = \oplus_{ \substack{k \geq p \\ k \in \bZ}} H_{{\scriptscriptstyle{BRST}}}^{p, k},
\end{equation}
where $H_{{\scriptscriptstyle{BRST}}}^{p, k}$ denotes the space of cohomology classes of function degree $k$ at total ghost number $p$. In particular, $H^{0, \bullet}_{d_{\scriptscriptstyle{BRST}}}$ is a graded commutative algebra;
\item The Poisson bracket $\{\cdot,\cdot\}$ on $\cA_{\scriptscriptstyle{BFV}}$, which is of function degree $-2$, descends to $H^{0, \bullet}_{{\scriptscriptstyle{BRST}}}$; keeping the same notation for the induced bracket in cohomology, 
we have that  
$$
(H^{0, \bullet}_{{\scriptscriptstyle{BRST}}}, \{\cdot, \cdot\})
$$ 
is a graded Poisson algebra of degree $-2$. (With the different grading mentioned at the end of $\S$ \ref{subsec:BFValg}, this result can also be regarded as a consequence of Prop.~\ref{prop:Qcoh}.)

\end{itemize}

Recall from $\S$ \ref{subsection: grd view hamilt actions} the coisotropic ideal $\cI = \SP{J^\sharp(\ov\g[-2])}\subseteq \Fun$ and the associated graded Poisson algebra of degree $-2$ given by $N(\cI)/\cI$. Consider the natural map
\begin{equation} \label{Map 0th cohomology and N(I)/I}
   \Phi:  H^{0, \bullet}_{{\scriptscriptstyle{BRST}}}  \longrightarrow \frac{N(\cI)}{\cI},  \quad [a^{0,0} + a^{1,1} + \cdots]  \longmapsto \overline{a^{0, 0}}
\end{equation}
which is well-defined by the same argument as in the proof of  \cite[Lemma~5.2]{Carvalho}.
The following result is analogous to \cite[Thm~5.2]{Carvalho} and proven similarly.

\begin{theorem}\label{thm: H0 BFV2} The map 
$\Phi:  H^{0, \bullet}_{{\scriptscriptstyle{BRST}}}\to  {N(\cI)}/{\cI}$
is a morphism of degree -2 Poisson algebras, and is an isomorphism when $0 \in \ov\g[-2]$ is a regular value of the moment map $J$ (i.e., when conditions (r1) and (r2) in $\S$ \ref{subsection: Courant red} are satisfied).
\end{theorem}

This result shows that the BFV algebra $(\cA_{\scriptscriptstyle{BFV}}, \{\cdot,\cdot\}, Q)$ extends the classical homological formulation of ordinary hamiltonian reduction \cite{Kostant, Stasheff} to the context of symplectic $\bN$-manifolds of degree $2$. 
We will now consider an enhancement of the BFV algebra in order to model the reduction of symplectic $\bN Q$-manifolds of degree 2.

\subsection{Extended BFV algebra} Suppose  that $(\ov\g, [\cdot,\cdot],\delta)$ is a dgla and that $(\cM, \{\cdot,\cdot\},\theta)$ is a hamiltonian $\ov\g$-space with moment map $J^\sharp$, as in $\S$ \ref{subsection: hmlt actions Courant alg}.
We now see how to incorporate $\delta$ and $\theta$ into the BRST algebra $(\cA_{\scriptscriptstyle{BFV}}, \{\cdot,\cdot\}, d_{\scriptscriptstyle{BRST}})$.

Regarding $\delta$, we define an element 
$$
\theta^{1,1}\in A^{1,1}\cap (\cA_{\scriptscriptstyle{BFV}})_3\subseteq \cA^{0,3}_{\scriptscriptstyle{BFV}}
$$
characterized by the property that $\{\theta^{1,1}, \cdot \, \}|_{\ov\g^*[-1]}: \ov\g^*[-1] \rightarrow \ov\g^*[-1]$ is the graded dual to the differential $\delta: \ov\g[-1] \rightarrow \ov\g[-1]$,
\begin{equation}\label{eq:deltavee}
\delta^\vee: \overline{\g}^*[-1]\to \overline{\g}^*[-1], \qquad \delta^\vee{\lambda} = - (-1)^{\m{\lambda}} \lambda \circ \delta,
\end{equation}
so that 
$$
\delta^\vee|_{\g^*}=\delta^*_\a, \qquad  \delta^\vee|_{\a^*}= -\delta_\h^*, \qquad  \delta^\vee|_{\h^*}= 0.
$$ 
In terms of bases $\{w^i\}$, $\{v^j\}$, $\{u^k\}$ for $\h, \a$ and $\g,$ respectively, we may write $\delta: \ov\g \rightarrow \ov\g$ as 
\begin{equation}\label{eq: delta in coordinates - proof}
    \delta w^i = a^i_jv^j \quad \text{and} \quad \delta v^j = b^j_k u^k, \quad a^i_j, b^j_k \in \mathbb{R},
\end{equation} 
and $\theta^{1,1}$ takes the explicit form 
\begin{equation}\label{eq: theta11 coordinates}
    \theta^{1, 1} =  - a^i_j w^*_i v^j + b^\ell_k v^*_\ell u^k.
\end{equation}
A direct verification shows that the above expression is independent of the choice of bases, and hence defines an element in $A^{1,1}\cap (\cA_{\scriptscriptstyle{BFV}})_3$.

Let
\begin{equation}\label{eq: courant cocyle dgla case}
    \Theta \coloneqq \theta +\theta^{1,1} \in \cA_{\scriptscriptstyle{BFV}}^{  0, 3},
\end{equation}

\begin{proposition}\label{prop: Courant cocycle dgla action}
The element $\Theta$ satisfies
$$
d_{\scriptscriptstyle{BRST}}\Theta = \{Q, \Theta \}=0 \quad \mbox{and} \quad \{\Theta, \Theta\} = 0.
$$
\end{proposition}

\begin{proof}
 We fix bases $\{w^i\}$, $\{v^j\}$, $\{u^k\}$ for $\h, \a$ and $\g$, and write $\theta^{1,1}$ as in \eqref{eq: theta11 coordinates}.

Since $\{\theta, \theta\} = 0$ and $a^i_j, b^\ell_k \in \mathbb{R}$,
we have 
$$
\{\Theta, \Theta\} = -2\{a_j^i w_i^* v^j, b^\ell_k v^*_\ell u^k\}=-2  a_j^i b^\ell_k \{ w_i^* v^j, v^*_\ell u^k \}
     = -2a_j^i b^\ell_k w^* _i \{v^j, v^*_\ell u^k \} =-2 a_\ell^i b^\ell_k w^*_i u^k,
$$ 
which vanishes because $\delta^2 = 0$. Hence $\{\Theta, \Theta\} = 0.$

We now show that $\{Q, \Theta\} = 0.$ Recall that $Q = Q^{1, 0} + Q^{2, 1} \in \cA^{1, 3}_{\scriptscriptstyle{BFV}}$, with $Q^{1, 0}$ and $Q^{2,1}$ as in \eqref{eq: Q10 comp BRST} and \eqref{eq: Q21 comp BRST}, respectively.
Then 
$$
 \{Q, \Theta\} = \{Q^{1,0}, \theta\} + \{Q^{1, 0}, \theta^{1, 1}\} + \{Q^{2, 1}, \theta\} + \{Q^{2, 1}, \theta^{1, 1}\}.
$$
Since $\theta \in \Fun$ and $Q^{2,1} \in S^2({\ov\g}^*[-1]) \otimes S^1({\ov\g}[-1])$, by the definition of the BFV bracket we have that $\{Q^{2,1},\theta\}=0$.

We claim that $\{Q^{1,0}, \theta\} + \{Q^{1, 0}, \theta^{1, 1}\} = 0$. 
This can be verified by a direct computation using the expressions \eqref{eq: Q10 comp BRST} and \eqref{eq: theta11 coordinates}, along with 
the differential-preserving condition \eqref{eq:diffpres}, which can be written in terms of \eqref{eq: delta in coordinates - proof} as
\begin{align}
    \{\theta, J^\sharp_0(w^i)\} & = J^\sharp_1 (\delta w^i) = a^i_jJ^\sharp_1(v^j), \label{ theta J0 = aJ1} \\
     \{\theta, J^\sharp_1(v^j)\} & = J^\sharp_2 (\delta v^j) = b^j_k J^\sharp_2 (u^k), \label{theta J1v = bJ2}\\
     \{\theta, J^\sharp_2(u^k)\} & = 0. \label{theta J2u = 0}
\end{align}

To conclude the proof, we must check the remaining condition $\{Q^{2, 1}, \theta^{1, 1}\} = 0$.
By \eqref{eq: def conditions for BRST charge} and \eqref{eq:deltavee},
the hamiltonian vector
fields of $Q^{2,1}$ and $\theta^{1,1}$ satisfy
\[
\{Q^{2,1},\cdot\}
\big|_{S^\bullet(\ov\g^*[-1])}
=
d_{\scriptscriptstyle{CE}}^{\ov\g}, \qquad \; \{\theta^{1,1},\cdot\}
\big|_{\ov\g^*[-1]} = \delta^\vee.
\]
We extend $\delta^\vee$ as a degree $1$ derivation of $S^\bullet(\ov\g^*[-1])$.

The condition that $\delta$ is a derivation of the graded Lie bracket,
\[
\delta[x,y]
=
[\delta x,y]
+
(-1)^{|x|}[x,\delta y],
\]
is equivalent to
\begin{equation*}\label{eq:CE-delta-anticommute}
d_{\scriptscriptstyle{CE}}^{\ov\g}\circ\delta^\vee + \delta^\vee\circ d_{\scriptscriptstyle{CE}}^{\ov\g} = 0.
\end{equation*}
For every 
$\lambda\in\ov\g^*[-1]$, the graded Jacobi identity gives
$$
\bigl\{\{Q^{2,1},\theta^{1,1}\},\lambda\bigr\}
=
\bigl\{Q^{2,1}, \{\theta^{1,1},\lambda\}\bigr\} + \bigl\{\theta^{1,1},
       \{Q^{2,1},\lambda\}\bigr\} = d_{\scriptscriptstyle{CE}}^{\ov\g}
\bigl(\delta^\vee\lambda\bigr)
+
\delta^\vee
\bigl(d_{\scriptscriptstyle{CE}}^{\ov\g}\lambda\bigr) =0.
$$
Since
\[
\{Q^{2,1},\theta^{1,1}\}
\in
S^2(\ov\g^*[-1])\otimes S^1(\ov\g[-1])
\]
and the canonical brackets with the linear ghost coordinates detect
its component in $\ov\g[-1]$, the preceding identity for every
$\lambda\in\ov\g^*[-1]$ implies that
\[
\{Q^{2,1},\theta^{1,1}\}=0.
\]
\end{proof}

We call the element $\Theta$ in \eqref{eq: courant cocyle dgla case} a {\em Courant cocycle}. It follows from the previous proposition that
\begin{equation}\label{eq: extended brst charge}
S \coloneqq Q + \Theta \in \cA^{1, 3}_{\scriptscriptstyle{BFV}} \oplus \cA^{0, 3}_{\scriptscriptstyle{BFV}} 
\end{equation}
satisfies the master equation $\{S, S\} = 0$. We shall refer to $S$ as the {\em extended BRST charge} and 
$$
(\cA_{\scriptscriptstyle{BFV}}, \{\cdot, \cdot\}, S)
$$  
as the {\em extended BFV algebra}. 

Note that the map $\Phi: H^{0, \bullet}_{{\scriptscriptstyle{BRST}}}  \to N(\cI)/\cI$ in \eqref{Map 0th cohomology and N(I)/I} satisfies
\begin{equation}\label{eq:PhiTheta}
\Phi([\Theta]) = \overline{\theta}.
\end{equation}
We therefore view the extended BFV algebra as the homological model for the hamiltonian reduction of the symplectic $\bN Q$-manifold $(\cM, \{\cdot, \cdot\}, \theta )$.

\section{The dg Courant-Dorfman algebra of the BFV model}\label{sec:proof}

In this section, we prove  Theorem~\ref{thm:main} by associating a dg Courant-Dorfman algebra structure to the extended BFV algebra of a  hamiltonian reduction in degree $2$.
This will be done via Theorem~\ref{thm: derived dg CD} once we fit BFV algebras into the framework of $\S$ \ref{subsec:bigraded}.

\subsection{Weights  on the BFV algebra}\label{subsec:CD-BFV}

Let $\cA_{\scriptscriptstyle{BFV}}$ be a BFV algebra as in $\S$ \ref{subsec:BFValg}.
We introduce a new grading on this algebra by defining the {\em weight} of an element $a\in \cA_{\scriptscriptstyle{BFV}}^{k, \ell} $ to be $wght(a) :=  \ell - k$, i.e., 
$$
wght = fdeg - tgh.
$$


With this grading, the generators of the BFV algebra coming from the Lie algebra $\ov\g$ have the following degrees:

\begin{center}
\begin{tabular}{c|c|c|c|c|c|c}

    & $\h$ & $\a$ & $\g$ & $\g^*$ & $\a^*$ & $\h^*$  \\
     \hline
     $wght$ & 0 & 1 & 2 & 0 & 1 & 2      
\end{tabular}
\hspace{0.1cm},
\end{center}
whereas on $\Fun$  weight coincides with function degree, since $tgh(\Fun)= 0$. It follows that $\cA_{\scriptscriptstyle{BFV}}$ is non-negatively graded with respect to $wght$.

We regrade $\cA_{\scriptscriptstyle{BFV}}$ by replacing the $tgh$ by $wght$, i.e., we consider
the bigraded algebra $\cA^{\bullet, \bullet}$ given by
\begin{equation}\label{eq:regrade}
\cA^{\ell - k, \ell} \coloneqq \cA_{\scriptscriptstyle{BFV}}^{k, \ell}.
\end{equation}

The BFV bracket on $\cA^{\bullet, \bullet}$ is of weight and function degree $-2$, so it is of bidegree $(-2, -2)$. In terms of the new grading, 
we have 
$$
Q \in \cA^{2, 3} =  \cA^{1, 3}_{\scriptscriptstyle{BFV}} \quad \text{and} \quad \Theta \in  \cA^{3, 3} = \cA^{0, 3}_{\scriptscriptstyle{BFV}}.
$$ 
Therefore (cf. Def. \ref{def:bigHPA}), 
$$
(\cA^{\bullet, \bullet},\{\cdot,\cdot\}, S=\Theta+Q)
$$ 
is a bigraded hamiltonian Poisson algebra of bidegree $(-2, -2)$. 

The differential $\{S,\cdot\}$ on $\cA^{\bullet, \bullet}$ decomposes as in diagram \eqref{eq: diagram dg CD}. The differentials along the diagonals are given by the action of the Courant cocycle \eqref{eq: courant cocyle dgla case}, whereas the vertical differentials correspond to the BRST differential, which preserves components of a fixed weight (as the BRST charge is of weight 2 and the BFV bracket of weight -2).

By Theorem \ref{thm: derived dg CD}, we know that the pair 
$$
(\cA^{0,\bullet},\cA^{1,\bullet})
$$
carries an induced dg Courant-Dorfman structure. Next, we describe it explicitly.

\subsection{Proof of Theorem \ref{thm:main}}

Recall the setup of Theorem \ref{thm:main}: a hamiltonian action $(\varphi, \varrho, \mu)$ of  a dgla $\ov\g$, described by  $(\g, \a, \h, \delta, \varpi)$, on a Courant algebroid $(E\to M, \SP{\cdot,\cdot}, \rho, \Cour{\cdot,\cdot})$; see $\S$ \ref{subsection: hmlt actions Courant alg}. 

We consider the corresponding graded-geometric formulation,
in which $(E\to M, \SP{\cdot,\cdot}, \rho, \Cour{\cdot,\cdot})$ is encoded by its associated symplectic $\bN Q$-manifold $(\cM, \{\cdot,\cdot\},\theta)$ (see $\S$~\ref{subsection: grd view courant alg}), endowed with a hamiltonian $\ov\g$-action with moment map
$$
J^\sharp: \ov\g[-2] \rightarrow \Fun
$$ 
such that $J_0^\sharp = \mu^*$, $J_1^\sharp = \varrho$ and $J_2^\sharp = \varphi$.

Consider the graded algebra $\cA_{\scriptscriptstyle{BFV}}$ \eqref{eq: bfv alg def}, and $\cA^{\bullet,\bullet}$ introduced in $\S$ \ref{subsec:CD-BFV}, see \eqref{eq:regrade}. At weight $0$ we find the graded algebra
\begin{equation}\label{eq: weight zero part}
\cR:= \cA^{0,\bullet} = C^\infty(M)\otimes \wedge \g^* \otimes \wedge \h,
\end{equation}
whereas at weight $1$ we find the graded $\cR$-module 
\begin{equation}\label{eq: weight one part}
\cE := \cA^{1,\bullet} =  \cR \otimes_{C^\infty(M)}(\Gamma(E[-1] \oplus \a_M^*[-2] \oplus \a_M)).
\end{equation}
The differentials $\dR: \cR^{\bullet } \rightarrow \cR^{ \bullet + 1}$ and  $\dE: \cE^{\bullet} \rightarrow \cE^{\bullet+1}$ are obtained by the restriction of the BRST differential $d_{\scriptscriptstyle{BRST}} = \{ Q, \, \cdot \,\}$ to $\cA^{0, \bullet}$ and $\cA^{1,\bullet}$, respectively. 

Using the coordinate expression for the BRST charge in Remark~\ref{rem: coord brst charge} and the definition of the BFV bracket, a direct computation shows that, for $f \in C^\infty(M),$ $ u\in \g,$ $ u^* \in \g^*,$ and $w \in \h,$ we have 
\begin{equation}\label{eq: weight zero brst diff}
\dR(f)(u)=\Lie_{\underline{\varphi}(u)}f,     \quad \dR(u^*) =  \dCE(u^*) \quad \text{and} \quad \dR(w) = \mu^*(w) + \dCE ^\h(w),    
\end{equation}
where $\underline{\varphi}(u)=\sigma(\varphi(u)) \in \mathfrak{X}(M)$ (see \eqref{eq:Mact}), while $\dCE$ and $\dCE ^\h$ denote the Chevalley-Eilenberg differentials for the $\g$-modules $\bR$ (trivial) and $\h$, respectively.

Similarly, $\dE$ satisfies 
\begin{equation*}\label{eq: weight one brst diff}
\dE(e)(u) =\varphi(u)e, \quad   \dE(e)(v) = - \SP{\varrho(v),e},
\quad \dE(v^*) = d_{\scriptscriptstyle{CE}}^{\a^*}(v^*), \quad \dE(v) =  \varrho(v) + \iota_v\varpi + d^\a_{\scriptscriptstyle{CE}}(v),
\end{equation*}
for $e \in \Gamma(E)$, $ u \in \g,$ $ v\in\a $ and $v^* \in \a^*$, with $d_{\scriptscriptstyle{CE}}^{\a^*}$ and $d_{\scriptscriptstyle{CE}}^{\a}$ the Chevalley-Eilenberg differentials corresponding to  the $\g$-modules $\a^*$ and $\a$.

The pairing $\SP{\cdot, \cdot}_{\scrsE}: \cE \otimes_{\cR} \cE \rightarrow \cR$ is 
the restriction of the BFV bracket to $\cA^{1, \bullet}$, hence given by
\begin{equation*}\label{eq: pairing - proof}
\SP{e_1, e_2}_{\scrsE} = \SP{e_1, e_2},  \quad  \SP{v^\prime, v^*}_{\scriptscriptstyle{\cE}} = v^*(v^\prime) 
\quad \text{and} \quad \SP{e, v^\prime}_{\scrsE} = 0 = \SP{e, v^*}_{\scrsE},
\end{equation*} 
for $e, e_1, e_2 \in \Gamma(E),$ $v^\prime\in \a$ and $v^* \in \a^*$.

The derivation and Leibniz bracket induced by the Courant cocyle \eqref{eq: courant cocyle dgla case} satisfy
\begin{itemize}
\item $\partial: \cR \rightarrow \cE$, 
\begin{equation}\label{eq: partial from Courant cocycle - proof}
\partial(u^*) = \delta^*(u^*), \quad \partial(w) = -\delta(w) \quad \text{and} \quad \partial(f) = \rho^*(df),
\end{equation}
for $u \in \g^*,$ $w \in \h$ and $f \in C^{\infty}(M);$

\item $\Cour{\cdot,\cdot}_{\scrsE}: \cE \otimes_{\bR} \cE \rightarrow \cE$, 
$$
\Cour{e_1, e_2}_{\scrsE} \coloneqq  \Cour{e_1, e_2} \quad \text{and} \quad \Cour{e, v}_{\scrsE} =  \Cour{e, v^*}_{\scrsE} = \Cour{v, v^*}_{\scrsE} = \Cour{v_1,v_2}_{\scrsE} = \Cour{v_1^*, v_2^*}_{\scrsE}=0
$$
for $e, e_1, e_2 \in \Gamma(E),$ $v, v_1,v_2 \in \a$ and $v^*,v_1^*,v_2^* \in \a^*.$ 
\end{itemize}

Since the pairing $\SP{\cdot, \cdot}_{\scrsE}$ is nondegenerate,  the derivation $\partial: \cR \rightarrow \cE$ is equivalent to an anchor map $\rhocE : \cE \rightarrow \mathrm{Der}(\cR)$ via the relation $\rhocE(\cdot)(r)=\SP{\, \cdot \,,\partial(r)}$, for $r\in \cR$. From \eqref{eq: partial from Courant cocycle - proof}, it follows that $\rhocE : \cE \rightarrow \mathrm{Der}(\cR)$ satisfies 
\begin{equation*}\label{eq: anchor from Courant cocyle - proof}
\rhocE (e) = \rho(e), \quad \rhocE(v) = \iota_\delta(v) \quad \text{and} \quad \rhocE(v^*) = \iota_\delta^*(v^*),
\end{equation*}
for $e \in \Gamma(E),$ $v\in \a$ and $v^*\in \a^*.$ 

Therefore, the pair $(\cR, \cE)$, with $\cR$ as in \eqref{eq: weight zero part} and $\cE$ as in \eqref{eq: weight one part}, is endowed with a dg Courant-Dorfman structure with differentials, pairing, bracket and anchor determined as above. This proves the first part of Theorem \ref{thm:main}. 

As for the second part, note that, by \eqref{eq:regrade}, we have the identification (of graded Poisson algebras)
$$
\oplus_kH_Q^{k,k}(\cA) = H^{0,\bullet}_{\scriptscriptstyle{BRST}}.
$$
If conditions (r1) and (r2) in $\S$ \ref{subsection: Courant red} hold, then by Theorem~\ref{thm: H0 BFV2} and \eqref{eq:PhiTheta}, the map $\Phi$ in \eqref{Map 0th cohomology and N(I)/I} defines an isomorphism
$$
(\oplus_kH_Q^{k,k}(\cA), \{\cdot,\cdot\}, [\Theta]) \cong (N(\cI)/\cI, \{\cdot,\cdot\}, \overline{\theta})
$$
of graded hamiltonian Poisson algebras of degree $-2$ and therefore an isomorphism of the corresponding Courant-Dorfman algebras:
$$
(H^{0,0}_Q(\cA) , H^{1,1}_Q(\cA)) \cong ((N(\cI)/\cI)_0, N(\cI)/\cI)_1).
$$
Since $(H^0_\cR, H^1_\cE) = (H^{0,0}_Q(\cA) , H^{1,1}_Q(\cA))$, by \eqref{eq:algreduc} we obtain an identification of Courant-Dorfman algebras
$$
(H^0_\cR, H^1_\cE) \cong (C^\infty(\mu^{-1}(0))^\g, \Gamma(K^\perp/K)^\g).
$$
Finally, following $\S$ \ref{subsection: Courant red}, when the hamiltonian action $(\varphi,\varrho,\mu)$ is regular (in particular, (r1) and (r2) are satisfied), we have a further identification of Courant-Dorfman algebras,
$$
(C^\infty(\mu^{-1}(0))^\g, \Gamma(K^\perp/K)^\g) \cong (C^\infty(M_{red}), \Gamma(E_{red})),
$$
which concludes the proof of Theorem~\ref{thm:main}.

\section{Graded geometric interpretation}\label{sec:interpret}

In the context of Theorem \ref{thm:main}, we may view  
 $\cR = C^\infty(M) \otimes \wedge \g^* \otimes \wedge \h$ as the algebra of functions on a $\mathbb Z$-manifold $\cN_0$,
and $\dR$ as a homological vector field  on $\cN_0$, so that 
$$
(\cN_0, \dR)
$$ 
is a dg manifold. 

We now explore this viewpoint: we first show that $(\cN_0,\dR)$ canonically gives rise to a bigraded hamiltonian Poisson algebra (in the sense of Def.~\ref{def:bigHPA}) on functions on the shifted cotangent bundle $T^*[2]T[1]\cN_0$, and then compare it with 
the BFV algebra $\cA_{\scriptscriptstyle{BFV}}$ in the exact case treated in $\S$ \ref{subsec:exact}

Although we are mostly concerned with the specific graded manifolds $\cN_0$ above, the discussion in $\S$ \ref{subsec:cotang} concerns general properties of $\bZ$-manifolds; we will follow coventions from \cite{Mehta's Thesis}, see also \cite{MJ,Fairon,KotovSalnikov,VoronovGradedGeometry}.

\subsection{The canonical bigraded  hamiltonian Poisson algebra of a dg manifold }\label{subsec:cotang}
Let $\cN_0$ be a $\bZ$-manifold and $\cR=C_{\mathrm{poly}}(\cN_0)$. We endow $\cN_0$  with a homological vector field $\dR$. 

Let $z^a$ be homogeneous local coordinates on $\cN_0$, and let
\[
z^a,\qquad \dot z^a
\]
be the induced coordinates on the shifted tangent bundle $T[1]\cN_0$, with
$|\dot z^a|=|z^a|+1.$

The {\em de Rham vector field} $d_{\mathrm{dR}}$ on $T[1]\cN_0$ is a homological vector field given locally by
\[
d_{\mathrm{dR}}
=
\dot z^a\frac{\partial}{\partial z^a},
\]
so that
$d_{\mathrm{dR}}z^a=\dot z^a$, and $
d_{\mathrm{dR}}\dot z^a=0$.

Let $X \in\mathfrak X(\cN_0):=\Der(\cR)$ be a homogeneous vector field of degree $|X|$, 
\[
X=X^a(z)\frac{\partial}{\partial z^a}.
\]
Then $X$ gives rise to two vector fields on $T[1]\cN_0$:

\begin{itemize}

\item The {\em contraction vector field} $\iota_X$ on $T[1]\cN_0$ is 
\[
\iota_X
=
X^a(z)\frac{\partial}{\partial\dot z^a}, \qquad |\iota_X| = |X|-1;
\]
It is characterized intrinsically by
\[
\iota_X(r)=0,
\qquad
\iota_X(d_{\mathrm{dR}}r)=X(r),
\]
for $r\in\cR$.

\item  The {\em Lie derivative} is the vector field
$
\cL_X:=[\iota_X,d_{\mathrm{dR}}]$;
i.e., 
\[
\cL_X
=
\iota_Xd_{\mathrm{dR}}
+
(-1)^{|X|}d_{\mathrm{dR}}\iota_X,  \qquad
|\cL_X|=|X|.
\]
Locally,
\[
\cL_X
=
X^a\frac{\partial}{\partial z^a}
+
(-1)^{|X|}
(d_{\mathrm{dR}}X^a)
\frac{\partial}{\partial\dot z^a},
\]
In particular,
\[
\cL_X(r)=X(r),
\qquad
\cL_X(d_{\mathrm{dR}}r)
=
(-1)^{|X|}
d_{\mathrm{dR}}(X(r)).
\]

\end{itemize}

The shifted cotangent bundle $T^*[2]T[1]\cN_0$ is endowed with a canonical Poisson bracket of degree $-2$. In canonical cotangent coordinates
$$
z^a, \quad \dot{z}^a, \quad p_{\dot{z}^a}, \quad p_{z^a},
$$
with $\m{\dot{z}^a}= \m{z^a}+1$, $\m{ p_{\dot{z}^a}}=-\m{z^a}+1$, $\m{p_{z^a}}=-\m{z^a}+2$, the canonical Poisson bracket satisfies
$$
\{p_{z^a}, z^b\}_{\mathrm{can}} =  \delta_{a}^b, \qquad \{p_{\dot{z}^a}, \dot{z}^b\}_{\mathrm{can}}=\delta_{a}^b,
$$
with all other brackets between distinct types of coordinates being zero. 

We define a second grading on $C_{\mathrm{poly}}(T^*[2]T[1]\cN_0)$, called {\em weight}, accounting for degree shifts on tangent and cotangent coordinates; it is determined by
\begin{equation}\label{eq:weight}
wght(z^a)=0, \quad wght(\dot{z}^a)=1, \quad wght(p_{\dot{z}^a})=1, \quad wght(p_{z^a})=2.
\end{equation}
We note that the weight grading is intrinsic, invariant by change of coordinates.

With respect to weight and the usual function degree, $\{\cdot,\cdot\}_{\mathrm{can}}$ is a Poisson structure on 
$$
C^{\bullet,\bullet}_{\mathrm{poly}}(T^*[2]T[1]\cN_0)
$$ 
of bidegree $(-2,-2)$. 

Let $\Theta_{\mathrm{dR}}$ and $Q_0$ be the fiber-linear hamiltonian functions for the lifts of the vector fields $d_{\mathrm{dR}}$ and $\cL_{\dR}$, characterized by
$$
\{\Theta_{\mathrm{dR}}, f\}_{\mathrm{can}} = d_{\mathrm{dR}} f, \quad \mbox{and} \quad
 \{Q_0, f\}_{\mathrm{can}} = \cL_{\dR} f,
$$
for $f\in C_{poly}(T[1]\cN_0)$ (naturally regarded as functions in $T^*[2]T[1]\cN_0$).

\begin{proposition}
$
(C^{\bullet,\bullet}_{\mathrm{poly}}(T^*[2]T[1]\cN_0), \{\cdot,\cdot\}_{\mathrm{can}}, \Theta_{\mathrm{dR}}+Q_0)
$
is a bigraded hamiltonian Poisson algebra, in the sense of Def.~\ref{def:bigHPA}. 
\end{proposition}

\begin{proof}
The de Rham vector field $d_{\mathrm{dR}}$ has weight $1$ and
function degree $1$, while $\cL_{\dR}$ has weight $0$ and
function degree $1$. Hence their cotangent hamiltonians satisfy
$$
\Theta_{\mathrm{dR}} \in C^{3,3}_{\mathrm{poly}}(T^*[2]T[1]\cN_0), \qquad {Q}_0 \in C^{2,3}_{\mathrm{poly}}(T^*[2]T[1]\cN_0).
$$ 

For vector fields $V,W$ on $T[1]\cN_0$, recall that their fiber-linear
hamiltonians $\ell_V$, $\ell_W$ satisfy
$\{\ell_V,\ell_W\}_{\mathrm{can}}\allowbreak=\ell_{[V,W]}$.
Therefore
\[
\{\Theta_{\mathrm{dR}},\Theta_{\mathrm{dR}}\}_{\mathrm{can}}=0,
\qquad
\{Q_{0},Q_{0}\}_{\mathrm{can}}=0, \qquad \{Q_{0},\Theta_{\mathrm{dR}}\}_{\mathrm{can}}=0
\]
as a consequence of $d_{\mathrm{dR}}^2=0$, $\dR^2=0$, and $[\cL_{\dR},d_{\mathrm{dR}}]=0$.
\end{proof}

The graded algebra of weight $0$ elements is $C_{\mathrm{poly}}(\cN_0)$, and we denote the module of weight $1$ by
$$
\bT_{\cN_0}:= C^{1,\bullet}_{\mathrm{poly}}(T^*[2]T[1]\cN_0).
$$
It follows from Theorem~\ref{thm: derived dg CD} that $(C_{\mathrm{poly}}(\cN_0), \bT_{\cN_0})$
has an induced structure of dg Courant-Dorfman algebra of degree $-1$, that we now describe.

For a vector field
$X\in\mathfrak X(\cN_0)$, denote by
$h_X$ the fiber-linear hamiltonian of the lift of the contraction vector field
$\iota_X$ on $T[1]\cN_0$, characterized by
\[
\{h_X,f\}_{\mathrm{can}} = \iota_Xf,
\]
for $f\in C(T[1]\cN_0)$. Note that $\m{h_X}=|X|+1$. In coordinates, for $X=X^a(z)\frac{\partial}{\partial z^a}$,
$$
h_X = X^ a(z) p_{\dot{z}^a}.
$$

We can see from \eqref{eq:weight} that any element in $\bT_{\cN_0}$ can be uniquely written in the form
$$
h_X + \widetilde{\alpha},
$$
where $\widetilde{\alpha} \in C_{\mathrm{poly}}(T[1]\cN_0)_{lin}$ is a fiber-linear function on $T[1]\cN_0$, locally of the form $A_a(z)\dot{z}^a$. So we have a canonical decomposition
$$
\bT_{\cN_0} = \{ h_X,\; X\in \mathfrak{X}(\cN_0) \}\oplus C_{\mathrm{poly}}(T[1]\cN_0)_{lin}.
$$

With respect to this decomposition, the dg Courant-Dorfman structure of degree $-1$ on $(C_{\mathrm{poly}}(\cN_0),\bT_{\cN_0})$ from Theorem~\ref{thm: derived dg CD} takes the following form:
For
\[
e_1=h_X+\widetilde\alpha,
\qquad
e_2=h_Y+\widetilde\beta,
\]
so that $\m{e_1}= \m{X}+1$, $\m{e_2}=\m{Y}+1$,
the internal differential on $\bT_{\cN_0}$ is
$$
d_{\bT_{\cN_0}}(e_1) = h_{[\dR,X]} + \cL_{\dR}\widetilde\alpha,
$$
the pairing is 
$$\SP{e_1,e_2} = \iota_X\widetilde\beta + (-1)^{|e_1||e_2|+1}
\iota_Y\widetilde\alpha,
$$
the derivation and corresponding anchor are
$$
\partial r = d_{\mathrm{dR}}r, \qquad \rho(e_1) = X, 
$$
and the bracket is given by
$$
\Cour{e_1,e_2} = h_{[X,Y]} + \cL_X\widetilde\beta - (-1)^{(|e_1|-1)(|e_2|-1)} \iota_Yd_{\mathrm{dR}}\widetilde\alpha.
$$
This is the ``geometric'' standard dg Courant-Dorfman algebra canonically associated with $(\cN_0,\dR)$.  

\begin{remark}
The relation between the standard dg Courant-Dorfman algebra $(C(\cN_0), \bT_{\cN_0})$ and the
cotangent bundle $T^*[2]T[1]\cN_0$, equipped with its canonical Poisson bracket and charge $\Theta_{\mathrm{dR}}+Q_0$ is the dg extension of the correspondence between standard Courant algebroids $\bT M$ and symplectic $\bN$-manifolds $T^*[2]T[1]M$ equipped with the canonical de Rham charge $\theta_M$, see \cite{RoytenbergThesis,Roytenberg}.
\hfill $\diamond$
\end{remark}

There is a direct relation with the  ``algebraic'' standard dg Courant-Dorfman algebra 
$$
(\cR, \bT_\cR[-1])
$$ 
of degree $-1$ described in Example~\ref{ex:shifted}, for $\cR=C_{\mathrm{poly}}(\cN_0)$,
in terms of  the shifted Cartan operators $\bm{i}_X$, $\bm{\Lie}_X$ and $\bm{d}$ from $\S$ \ref{subsection:shifts}, for $m=1$. We denote by $s: \Der(\cR)\to \Der(\cR)[-1]$ the operator that shifts degree by $1$.

We have an isomorphism of $\cR$-modules 
$$
\jmath: \Omega^1_\cR[-1] \to  C_{\mathrm{poly}}(T[1]\cN_0)_{lin},
$$
uniquely characterized by the condition
$$
\iota_{X_0}\jmath(\alpha) = \bm{i}_X\alpha
$$
where $X=s(X_0)$. This defining relation for $\jmath$, together with Cartan identities, implies that 
$$
\jmath(\bm{d}r) = d_{\mathrm{dR}}r, \qquad \cL_{X_0}\jmath(\alpha) = \jmath (\bm{\Lie}_X\alpha), \qquad \iota_{Y_0}d_{\mathrm{dR}}\jmath(\alpha) = \jmath(\bm i_Y\bm d\alpha).
$$
From the middle identity we see that
\[
\cL_{\dR}\jmath(\alpha)=\jmath(d_\Omega\alpha).
\]
With these identities, it directly follows that the map
\begin{equation}\label{eq:geomalgstd}
 \bT_\cR[-1]\to \bT_{\cN_0}, \qquad (X,\alpha)\mapsto h_{X_0}+ \jmath(\alpha)
\end{equation}
intertwines the pairing, derivation, anchor, bracket, and internal
differential, thereby establishing an isomorphism of dg Courant-Dorfman algebras of degree $-1$.

\subsection{Relation with $\cA_{\scriptscriptstyle{BFV}}$ in the exact case}\label{subsec:compare}
Back to the setup of Theorem~\ref{thm:main}, we have a hamiltonian action of a dgla $\ov\g = \h[2]\oplus \a[1]\oplus \g$ on a Courant algebroid $E\to M$, and consider the corresponding BFV algebra $\cA_{\scriptscriptstyle{BFV}}$ as a bigraded hamiltonian Poisson algebra of bidegree $(-2,-2)$ with respect to weight and function degrees, as in $\S$ \ref{subsec:CD-BFV}.

We regard $\cR = C^\infty(M) \otimes \wedge \g^* \otimes \wedge \h$ as the algebra of functions on a $\mathbb Z$-manifold $\cN_0$,
\begin{equation*}\label{eq:Rfun}
C_{\mathrm{poly}}(\cN_0)=\cR,
\end{equation*}
and $\dR$ as a  homological vector field on $\cN_0$.
The BFV algebra is
$$
\cA_{\scriptscriptstyle{BFV}} = \Fun\otimes S({\ov\g}^*[-1]) \otimes S(\ov\g[-1]),
$$
with degrees of generators given by
\[
\begin{array}{c|cccccc}
 & \h & \a & \g & \g^* & \a^* & \h^*\\
\hline
\text{degree}
&-1&0&1&1&2&3.
\end{array}
\]
It is equipped with the BFV bracket and the charge $\Theta + Q$.

From a graded-geometric perspective, $\cA_{\scriptscriptstyle{BFV}}$ can be seen as the graded Poisson algebra of (fiberwise polynomial) functions on 
$\cM \times T^*[2](\overline{\g}[1])$,
$$
\cA_{\scriptscriptstyle{BFV}} = C_{\mathrm{poly}}(\cM \times T^*[2](\overline{\g}[1]))
$$
where $\overline{\g}[1]$ is viewed as a linear graded manifold with $C(\overline{\g}[1])= S({\ov\g}^*[-1])$.

Let us assume that
\begin{itemize}
\item[(a1)] the dgla $\ov\g$ is determined by $(\g,\h,\Upsilon)$ as in $\S$ \ref{subsec:exact};
\item[(a2)] the Courant algebroid $E$ is of the form $\bT_\eta M$ as in Example~\ref{ex:exactCA}.
\end{itemize}

By condition (a1), $\a=\h\oplus \g$. If we set $\mathcal{W}:= \g[1]\oplus \h^*[-1]$, then its algebra of functions is
$$
C(\mathcal{W}) = \wedge \g^* \otimes \wedge \h, 
$$
and it follows that
\begin{equation}\label{eq:N0prod}
\cN_0 = M \times \mathcal{W}.
\end{equation}

By condition (a2), the symplectic $\bN$-manifold corresponding to $E$ is $\cM= T^*[2]T[1]M$, and therefore
\begin{equation}\label{eq:shiftedcot}
\cA_{\scriptscriptstyle{BFV}} = C_{\mathrm{poly}}(T^*[2]T[1]M \times T^*[2] (\overline{\g}[1])).
\end{equation}

 For $\xi\in\g^*$ and $w\in\h$, let 
$$
c_\xi, \qquad b_w
$$
denote the corresponding linear functions on $\mathcal{W}$. Their degrees are
$|c_\xi|=1$,
and $|b_w|=-1$.

Let
\[
\dot c_\xi,\qquad \dot b_w
\]
denote the corresponding fiber-linear functions on $T[1]W$, so that
$|\dot c_\xi|=2,$ and $|\dot b_w|=0$.

For $u\in\g$ and $\lambda\in\h^*$, let
\[
p_{c,u},\quad p_{\dot c,u},
\quad
p_{b,\lambda},\quad p_{\dot b,\lambda}
\]
be functions on $T^*[2]T[1]W$, with respective degrees $1$, $0$, $3$, $2$, characterized by
\begin{equation}\label{eq:intrinsic-cotangent-pairings}
\begin{aligned}
\{p_{c,u},c_\xi\}_{\mathrm{can}}
&=\xi(u),
&
\{p_{\dot c,u},\dot c_\xi\}_{\mathrm{can}}
&=\xi(u),
\\
\{p_{b,\lambda},b_w\}_{\mathrm{can}}
&=\lambda(w),
&
\{p_{\dot b,\lambda},\dot b_w\}_{\mathrm{can}}
&=\lambda(w),
\end{aligned}
\end{equation}
with all brackets between distinct coordinate--momentum pairs equal
to zero. 
In other words, these functions are ``cotangent momenta'' corresponding to $c_\xi$, $b_w$, $\dot{c}_\xi$ and $\dot{b}_w$.

We obtain a linear symplectic isomorphism
\begin{equation}\label{eq:linsymp}
 T^*[2](\ov\g[1])  \stackrel{\sim}{\longrightarrow}  T^*[2]T[1]\mathcal{W},
\end{equation}
by the following correspondence of generators:
\begin{equation}\label{eq:gencorrespond}
\begin{aligned}
c_\xi&\longleftrightarrow \xi\in \g^*,
&
p_{c,u}&\longleftrightarrow u \in \g,
\\
\dot c_\xi&\longleftrightarrow (0,\xi)\in \a^*=\h^*\oplus \g^*,
&
p_{\dot c, u}&\longleftrightarrow (0,u)\in \a=\h\oplus \g,
\\
b_w&\longleftrightarrow w \in \h,
&
p_{b, \lambda}&\longleftrightarrow \lambda\in \h^*,
\\
\dot b_w&\longleftrightarrow-(w,0)\in \a=\h\oplus \g,
&
p_{\dot b, \lambda}&\longleftrightarrow (\lambda,0)\in \a^*=\h^*\oplus \g^*.
\end{aligned}
\end{equation}
(The sign in the fourth line is forced by the fact that
$\{(\lambda,0),(w,0)\}_{\scriptscriptstyle{\mathrm{BFV}}} = -\lambda(w)$.)

By \eqref{eq:N0prod}, we obtain an induced symplectic isomorphism
$$
T^*[2]T[1]M \times T^*[2](\ov\g[1]) \cong T^*[2]T[1]M \times T^*[2]T[1]\mathcal{W} = T^*[2]T[1]\cN_0.
$$

At the level of functions, by \eqref{eq:shiftedcot}, we have
an isomorphism of graded Poisson algebras
\begin{equation}\label{eq:Psi}
\Psi:
\cA_{\scriptscriptstyle{BFV}} \stackrel{\sim}{\longrightarrow} C_{\mathrm{poly}}\bigl(T^*[2]T[1]\cN_0\bigr)
\end{equation}
Since $\Psi$ also preserves weights, it is an isomorphism of bigraded Poisson algebras.

We now compare the respective charges. 

We assume henceforth that we have a {\em standard hamiltonian action} in the sense of  $\S$ \ref{subsec:exact}: $\Upsilon=0$ (the dgla corresponds to a hemisemidirect product), $\eta=0$ ($E$ is the standard Courant algebroid) and the hamiltonian action is defined by $(\psi, \mu, \zeta)$ with $\zeta=0$.

\begin{proposition}
For a standard hamiltonian action, the isomorphism $\Psi$ satisfies
$$
 \Psi(\Theta) = \Theta_{\mathrm{dR}} , \quad \mbox{and} \quad \Psi(Q) = Q_0,
$$ 
i.e., $\Psi$ is an isomorphism of bigraded hamiltonian Poisson algebras.
\end{proposition}

\begin{proof} For the standard Courant algebroid, the corresponding Courant function $\theta_M$ is the hamiltonian of the de Rham vector field on $T[1]M$; locally,
$$
\theta_M = \dot{x}^i p_{x^i}.
$$
So, by \eqref{eq: courant cocyle dgla case}, $\Psi(\Theta) =\theta_M + \Psi(\theta^{1,1})$. We will describe $\Psi(\theta^{1,1})$ .

Choose bases $\{u_l\}$ of $\g$ and $\{w_i\}$ of $\h$. Denote the elements of the induced basis on $\a=\h\oplus \g$ by
\[
v_{\h,i}:=(w_i,0),
\qquad
v_{\g, l}:=(0,u_l).
\]
Denote the dual bases by $\{ u^{* l} \}$, $\{ w^{*i} \}$, and $\{ v^{*i}_{\h}, v^{*l}_{\g} \}$.

Since 
\[
\delta_{\h}(w_i)=v_{\h,i},
\qquad
\delta_{\a}(v_{\h,i})=0,
\qquad
\delta_{\a}(v_{\g,l})=u_l,
\]
it follows that 
\begin{equation}\label{eq:theta11-split}
\theta^{1,1}
=
- w^{*i}v_{\h,i}
+
v_{\g}^{*l}u_l.
\end{equation}

Using the correspondence \eqref{eq:gencorrespond}, we have
$$
c^l=\Psi(u^{*l}), \quad \dot{c}^l=\Psi(v^{*l}_\g), \quad  b^i=\Psi(w_i), \quad \dot{b}^i = - \Psi(v_{\h,i}).
$$
Using the  BFV bracket relations
\[
\{u_l,u^{*m}\}_{\scriptscriptstyle{\mathrm{BFV}}}
=
\delta_l^m,
\qquad
\{w^{*i},w_j\}_{\scriptscriptstyle{\mathrm{BFV}}}
=
\delta^i_j,
\]
we also see that
$$
p_{c^l} =  \Psi(u_l), \qquad p_{b^j}=\Psi(w^{*j}).
$$
It follows that 
$$
\Psi(\theta^{1,1}) = \dot c^l\,p_{c^l} +  \dot b^j\,p_{b^j}
$$
and therefore, locally, 
$$
\Psi(\Theta) = \dot{x}^i p_{x^i} +  \dot c^l\,p_{c^l} +  \dot b^j\,p_{b^j} = \Theta_{\mathrm{dR}}.
$$

We now consider $\Psi(Q)$.
For $r\in\cR= C(\cN_0)$,
$$
\{\Psi(Q),r\}_{\mathrm{can}} =
\Psi \bigl(\{Q,r\}_{\scriptscriptstyle{\mathrm{BFV}}}\bigr) =
d_{\cR}r = \cL_{\dR}(r).
$$
(Here we use that $\Psi(r)=r$, for $\cR=C(\cN_0)$ viewed as the weight-zero subalgebra in $\cA_{\scriptscriptstyle{BFV}}$ and as functions on $T^*[2]T[1]\cN_0$ via the natural projection).
Now, since
\[
{d}_{\mathrm{dR}}r
=
\{\Theta_{\mathrm{dR}},r\}_{\mathrm{can}},
\]
and $\{\Psi(Q),\Theta_{\mathrm{dR}}\}_{\mathrm{can}}= \{\Psi(Q),\Psi(\Theta)\}_{\mathrm{can}}=0$, by the graded Jacobi identity we obtain 
$$
\{\Psi(Q), {d}_{\mathrm{dR}}r\}_{\mathrm{can}}
=
-\{\Theta_{\mathrm{dR}},{d}_{\cR}r\}_{\mathrm{can}}
=
- {d}_{\mathrm{dR}}(d_{\cR}r)
=
\cL_{\dR}({d}_{\mathrm{dR}}r).
$$
It follows that
$$
\{\Psi(Q),f\}=\cL_{\dR}f
$$
for all functions $f$ on $T[1]\cN_0$.

We can also use the local expression for $Q= Q^{1,0}+Q^{2,1}$ to verify that $\Psi(Q)$ 
 is fiber-linear in the cotangent momenta and vanishes on the zero section. Since its hamiltonian vector field restricts to $\cL_{\dR}$ on $T[1]\cN_0$, 
 it follows from the uniqueness of the fiber-linear hamiltonian of a cotangent lift
 that $\Psi(Q)=Q_0$.
\end{proof}



\begin{remark} 
The isomorphism $\Psi: \cA_{\scriptscriptstyle{BFV}} \stackrel{\sim}{\longrightarrow} C_{{poly}}\bigl(T^*[2]T[1]\cN_0\bigr)$ of bigraded hamiltonian Poisson algebras restricts to an isomorphism of the corresponding dg Courant-Dorfman algebras defined by  weight $1$ $\cR$-modules,
$$
\Psi_1: \cE\stackrel{\sim}{\longrightarrow} \bT_{\cN_0}.
$$
As seen in $\S$ \ref{subsec:exact},  we have an isomorphism $\kappa: \cE\to \bT_\cR[-1]$, see \eqref{eq:kappa}, to the standard dg Courant-Dorfman algebra of degree $-1$. The composition $\Psi_1\circ\kappa^{-1}: \bT_\cR[-1] \to \bT_{\cN_0}$ coincides with the identification \eqref{eq:geomalgstd}.

\hfill $\diamond$
\end{remark}

\begin{example}[Relation with the Kalkman model]

Consider the special case of (a1) where $\h=0$, so that $\a=\g$ and the hamiltonian action on $\bT M$ is determined by a
Lie algebra action $\psi:\g\longrightarrow\mathfrak X(M)$.
Let
\[
A:=\g\ltimes M
\]
be the corresponding action Lie algebroid. 

In this case $\cN_0=A[1]=M\times\g[1]$, and
\[
\cR = C_{\mathrm{poly}}(\cN_0) = C^\infty(M)\otimes\wedge\g^*.
\]
The differential $\dR$ is the Lie algebroid differential of $A$, a homological vector field
$\cQ_A$ on $\cN_0$.

Choose a basis $\{e_l\}$ of $\g$, with dual degree $1$ coordinates
$\{c^l\}$, and write
\[
[e_m,e_n]=C^l_{mn}e_l,
\qquad
\psi_l:=\psi(e_l).
\]
In local coordinates $x^i$ on $M$, 
\[
\cQ_A = c^l\psi_l^i(x)\frac{\partial}{\partial x^i}
- \frac12C^l_{mn}c^mc^n \frac{\partial}{\partial c^l}.
\]

On $T[1]\cN_0$, set
\[
\xi^i:=d_{\mathrm{dR}}x^i,
\qquad
\mu^l:=d_{\mathrm{dR}}c^l.
\]
Their degrees are $|x^i|=0$, $|c^l|=1$, 
$|\xi^i|=1$, $|\mu^l|=2$.
It follows that
\begin{equation}\label{eq:Weil-action-algebroid}
C_{\mathrm{poly}}(T[1]\cN_0)
\cong
\Omega(M)\otimes\wedge\g^*\otimes S\g^*
=
\Omega(M)\otimes W(\g),
\end{equation}
which is the Weil algebra of the action Lie algebroid
$A=\g\ltimes M$ \cite{AbadCrainic, Mehta's Thesis}; here $W(\g)$ is the usual Weil algebra of $\g$. 

The de Rham vector field on $T[1]\cN_0$ is characterized by
\[
d_{\mathrm{dR}}x^i=\xi^i,
\qquad
d_{\mathrm{dR}}c^l=\mu^l,
\qquad
d_{\mathrm{dR}}\xi^i=0,
\qquad
d_{\mathrm{dR}}\mu^l=0,
\]
and recall that the tangent lift $\cL_{\cQ_A}$ is determined by
$\cL_{\cQ_A}(r)=\cQ_A(r)$, $[\cL_{\cQ_A},d_{\mathrm{dR}}]=0$.
In coordinates, we have
\begin{align*}
\cL_{\cQ_A}x^i &= c^l\psi_l^i,
& \cL_{\cQ_A}c^l &=
-\frac12C^l_{mn}c^mc^n,
\\
\cL_{\cQ_A}\xi^i &= -\mu^l\psi_l^i + c^l\frac{\partial\psi_l^i}{\partial x^j}\xi^j, &
\cL_{\cQ_A}\mu^l &= -C^l_{mn}c^m\mu^n.
\end{align*}

Consider the homological vector field
\[
d_{\mathrm K} := d_{\mathrm{dR}}+\cL_{\cQ_A}.
\]
Its action on generators is
\begin{align*}
d_{\mathrm K}x^i &= \xi^i+c^l\psi_l^i, &
d_{\mathrm K}c^l &= \mu^l-\frac12C^l_{mn}c^mc^n, \\
d_{\mathrm K}\xi^i &= -\mu^l\psi_l^i +c^l\frac{\partial\psi_l^i}{\partial x^j}\xi^j, &
d_{\mathrm K}\mu^l &= -C^l_{mn}c^m\mu^n.
\end{align*}

Under the identification \eqref{eq:Weil-action-algebroid}, this is the
Kalkman differential \cite{Kalkman}
\begin{equation}\label{eq:Kalkman-differential}
d_{\mathrm K} := d_M+d_W + c^l\cL_{\psi_l} - \mu^l\iota_{\psi_l},
\end{equation}
where the Weil differential is
\[
d_Wc^l
=
\mu^l-\frac12C^l_{mn}c^mc^n,
\qquad
d_W\mu^l
=
-C^l_{mn}c^m\mu^n.
\]

Thus
\[
\left(
C_{\mathrm{poly}}(T[1]\cN_0),
d_{\mathrm{dR}}+\cL_{\cQ_A}
\right)
\cong
\left(
\Omega(M)\otimes W(\g),
d_{\mathrm K}
\right).
\]
The basic subcomplex of this $\g$-differential graded algebra is the
Kalkman BRST model for equivariant cohomology, which is isomorphic to the Cartan model.

Hence, in our homological model for the reduction of the standard Courant algebroid $\bT M$, the extended BRST charge is obtanied by a cotangent lift of Kalkman's differential.
\hfill $\diamond$
\end{example}

\appendix

\section{Dg Courant-Dorfman algebras on $\Der(\cR)\oplus \Omega^1_\cR$}\label{app:twisted}

We will describe  examples of dg Courant-Dorfman algebras naturally associated with commutative dg algebras, analogous to the standard Courant algebroid of a manifold. For that we briefly review the Cartan calculus in this setting.

\subsection{Dg Cartan calculus}\label{subsec:dgcartan}
Let $\cR$ be a commutative graded algebra. Let $\Der(\cR)=\oplus_{k\in \mathbb{Z}}\Der^k(\cR)$ be the graded $\cR$-module of derivations of $\cR$. The graded commutator of derivations,
$$
[X,Y]= X\circ Y - (-1)^{\m{X}\m{Y}}Y\circ X,
$$
is a graded Lie bracket on $\Der(\cR)$
satisfying
\begin{equation}\label{eq:leibder}
[X,rY]= X(r) Y + (-1)^{\m{r}\m{X}} r[X,Y]
\end{equation}
for $r\in \cR$.

We consider the graded $\cR$-module of graded $k$-forms
$$
\Omega^k_\cR:= \mathrm{Hom}_{\cR}(\wedge^k \Der(\cR),\cR).
$$
We write 
\[ (\Omega^k_\cR)^\ell = \Omega_\cR^{k,\ell} 
\] 
for the space of $k$-forms of internal degree $\ell$. Thus $\alpha\in\Omega_\cR^{k,\ell}$ satisfies 
\[ 
\left| \alpha(X_1,\ldots,X_k) \right| = \ell+|X_1|+\cdots+|X_k|. 
\] 
We refer to $\ell$ as the \emph{internal degree} of $\alpha$.

For our purposes, it will suffice to model differential forms as duals of derivations (with no need for  Kahler differentials).

To set our conventions, we recall the basic operations for the Cartan calculus in the graded context:

\begin{itemize}
\item For a homogeneous $X\in \Der(\cR)$, we define $i_X: \Omega^k_\cR \to \Omega^{k-1}_\cR$ by
$$
(i_X\alpha)(X_1,\ldots,X_{k-1})= (-1)^{\m{X}\m{\alpha}}\alpha(X,X_1,\ldots,X_{k-1}).
$$
Note that $\m{i_X\alpha} = \m{X}+\m{\alpha}$, so $i_X$ is an operator of degree $\m{X}$. For $r\in \cR$, we have
\begin{equation}\label{eq:contract}
i_{rX}\alpha = r (i_X\alpha), \qquad i_X(r\alpha)=(-1)^{\m{X}\m{r}}r (i_X\alpha).
\end{equation}
\item For a homogeneous $X\in \Der(\cR)$, we define
$\Lie_X: \Omega^k_\cR\to \Omega^k_\cR$ by
\begin{align*}
(\Lie_X\alpha)(Y_1,\ldots,Y_k) =& X (\alpha(Y_1,\ldots,Y_k))\\& - (-1)^{\m{\alpha}\m{X}}\sum_{i=1}^k(-1)^{\m{X}(\m{Y_1}+\ldots\m{Y_{i-1}})}\alpha(Y_1,\ldots,[X,Y_i],\ldots,Y_k),
\end{align*}
which is an operator of internal degree $\m{X}$ satisfying
\begin{equation}\label{eq:Lieleib}
\Lie_X(r\alpha)= X(r) \alpha + (-1)^{\m{X}\m{r}} r \Lie_X\alpha,
\end{equation}
for $r\in \cR$.

\item We define $d:\Omega^k_\cR\to \Omega_\cR^{k+1}$ recursively by 
$$
i_Xd\alpha = \Lie_X\alpha - d(i_X\alpha),
$$
which is an operator of internal degree $0$. In particular,
$d:\cR\to \Omega^1_\cR$ is defined by $i_Xdr = X(r)$, and satisfies
\begin{equation}\label{eq:der0}
d(r_1r_2) = (dr_1)r_2 + r_1(dr_2).
\end{equation}

\end{itemize}

These operators satisfy graded versions of the usual Cartan identities:
\begin{equation}\label{eq:cartan1}
i_Xi_Y=-(-1)^{\m{X}\m{Y}}i_Yi_X, \quad [\Lie_X,i_Y]=\Lie_Xi_Y-(-1)^{\m{X}\m{Y}}i_Y\Lie_X=i_{[X,Y]},
\end{equation}
as well as
\begin{equation}\label{eq:cartan2}
[\Lie_X,\Lie_Y]=\Lie_{[X,Y]}, \quad\; \Lie_X=i_Xd+di_X, \quad\; \Lie_Xd=d\Lie_X, \quad\; d^2=0.
\end{equation}

Suppose that $(\cR, \dR)$ is a commutative dg algebra, i.e., $\dR\in \Der^1(\cR)$ and $[\dR,\dR]=2 \dR^2=0 $. We have natural lifts of $\dR$ to $\Der(\cR)$ and $\Omega^k_\cR$, namely
$$
d_{\Der}:= [\dR, \cdot] :  \Der^\bullet(\cR)\to \Der^{\bullet+1}(\cR)
$$
is a differential making $\Der(\cR)$ a dg $\cR$-module, while
$$
d_\Omega:= \Lie_{\dR}: \Omega_\cR^{k, \bullet}\to \Omega_\cR^{k,\bullet+1}
$$
makes $\Omega_\cR^k$ a dg $\cR$-module. 

By the graded Jacobi identity for the graded commutator of derivations we have that
\begin{equation}\label{eq:dDerbrk}
d_{\Der}[X,Y]
=
[d_{\Der}X,Y]
+
(-1)^{|X|}[X,d_{\Der}Y],
\end{equation}
while the Cartan identities immediately imply that
\begin{align}
d_{\Omega}(i_X\alpha)
&=
i_{d_{\Der}X}\alpha
+
(-1)^{|X|}i_X(d_{\Omega}\alpha), \label{eq:dgcartan1} \\ 
d_{\Omega}(\Lie_X\alpha)
&=
\Lie_{d_{\Der}X}\alpha
+
(-1)^{|X|}\Lie_X(d_{\Omega}\alpha),\label{eq:dgcartan2}\\
d_{\Omega}(d\alpha)
&=
d(d_{\Omega}\alpha).\label{eq:dgcartan3}
\end{align}

\subsection{Standard dg Courant-Dorfman algebras and twists}\label{subsec:twists}
Consider the dg $\cR$-module
$$
\mathbb{T}_\cR:= \Der(\cR)\oplus \Omega^1_\cR, \qquad d_{\bT}(X,\alpha)=(d_\Der(X), d_\Omega(\alpha)).
$$
Consider the pairing 
\begin{equation}\label{eq:pairingstd}
\SP{e_1,e_2} = i_X\beta + (-1)^{\m{e_1}\m{e_2}}i_Y\alpha,
\end{equation}
for $e_1=(X,\alpha)$, $e_2=(Y,\beta)$, the
bracket
\begin{equation*}
\Cour{(X,\alpha), (Y,\beta)}=([X,Y], \Lie_X\beta - (-1)^{\m{e_1}\m{e_2}}i_Yd\alpha),\label{eq:brktsdt app}
\end{equation*}
and let $\partial: \cR \to \bT_\cR$ be given by
\begin{equation}\label{eq:partialstd}
\partial r = (0, dr). 
\end{equation}
Then $\SP{\cdot,\cdot}$, $\partial$ and $\Cour{\cdot,\cdot}$ make $(\cR,\bT_\cR)$ a dg Courant-Dorfman algebra of degree $0$, referred to as the {\em standard Courant-Dorfman algebra} associated with $\cR$. More generally, we have the following class of examples of dg Courant Dorfman algebras.

\begin{proposition}\label{prop:twisted}
Let $F\in \Omega_\cR^{2,1}$ and $H\in \Omega_\cR^{3,0}$ be homogeneous forms satisfying
\begin{equation}\label{eq:FHclos}
d_\Omega F=0, \quad dH=0, \quad d_\Omega H=dF,
\end{equation}
and define $d_{\bT,F}: \bT_\cR\to \bT_\cR$ by
\begin{equation}\label{eq:dF}
d_{\bT,F}(X,\alpha) = (d_\Der(X), d_\Omega(\alpha) + (-1)^{\m{X}}i_XF).
\end{equation}
Then $d_{\bT,F}$ makes $\bT_\cR$ into a dg module over $\cR$ in such a way that the pairing \eqref{eq:pairingstd}, derivation \eqref{eq:partialstd} and bracket
\begin{equation}\label{eq:H-twisted-bracket}
\Cour{(X,\alpha), (Y,\beta)}_H=([X,Y], \Lie_X\beta - (-1)^{\m{e_1}\m{e_2}}i_Yd\alpha + (-1)^{\m{X}\m{Y}}i_Yi_XH)
\end{equation}
define a dg Courant-Dorfman algebra of degree $0$.
\end{proposition}

 \begin{proof} 
 The $\cR$-bilinearity of the pairing follows from \eqref{eq:contract}, and \eqref{eq:der0} ensures that $\partial$ is a (degree $0$) derivation.
 
 For a homogeneous derivation $X$, set
 \[ 
 c_F(X):=(-1)^{|X|}i_XF, \qquad  b_H(X,Y):=(-1)^{|X||Y|}i_Yi_XH.
 \] 
 By our contraction convention, the sign factors above ensure that $c_F$ and $b_H$ are defined by ordinary insertions into $F$ and $H$,
$$
c_F(X)(Y)=F(X,Y), \qquad b_H(X,Y)(Z)=H(X,Y,Z).
$$
We divide the verification of the axioms in Def.~\ref{def:dgDC} into steps.

\smallskip

\noindent{\em Step 1: $(\bT_\cR, d_{\bT,F})$ is a dg module over $\cR$.}
Note that $d_{\bT,F}$ has degree $1$, while the $\cR$-linearity 
$$
c_F(rX) = (-1)^{|X|+\m{r}}i_{rX}F  = (-1)^{\m{r}}rc_F(X)
$$ 
implies that \eqref{eq:leib2} is satisfied.

By \eqref{eq:dgcartan1}, we obtain
$$
d_\Omega c_F(X)+c_F(d_{\Der}X) = (-1)^{|X|} d_\Omega(i_XF) + (-1)^{|X|+1}i_{d_{\Der}X}F= i_X(d_\Omega F), 
$$
which implies that $d_{\mathbb T,F}^{2}(X,\alpha) = \bigl(0,i_X(d_\Omega F)\bigr)$. Hence 
$d_{\bT,F}^{2}=0$ follows from $d_\Omega F=0$.

\smallskip

\noindent{\em Step 2: Conditions (ii.) and (iii.).} The standard differential
$d_\bT$ satisfies (ii.), and the fact that this also holds for $d_{\bT,F}$ reduces to the identity
$$
\SP{(0,c_F(X)),e_2} + (-1)^{\m{e_1}}\SP{e_1,(0, c_F(Y))} = F(X,Y) + (-1)^{\m{X}\m{Y}} F(Y,X) =0,
$$
for homogeneous $e_1=(X,\alpha)$, $e_2=(Y,\beta)$.

Note that (iii.) holds for $d_\bT$ by \eqref{eq:dgcartan3}, and it holds for $d_{\bT,F}$ since
$$
d_{\bT,F}\partial r = d_\bT \partial r = \partial(\dR r).
$$
\smallskip

\emph{Step 3: Courant--Dorfman identities (1)--(5).} 
Axiom (1) follows from the Cartan identities \eqref{eq:cartan1} and the graded alternation of $H$. Axiom (2) uses Cartan's formula and the graded alternation of $H$. Axiom (3) is a consequence of $d^2=0$. Axiom (4) follows from \eqref{eq:leibder},  \eqref{eq:contract}, and \eqref{eq:Lieleib}  together with the $\cR$-multilinearity of $H$. Axiom (5) just uses the definition of the pairing.

\smallskip

\emph{Step 4: The Jacobi-like identity in (c).} This can be verified using the various Cartan identities \eqref{eq:cartan1} and \eqref{eq:cartan2} and $dH=0$.

\smallskip

\emph{Step 5: Condition (i.)}. The fact that the standard differental and bracket satisfy (i.) follows from \eqref{eq:dDerbrk}--\eqref{eq:dgcartan3}.
Since $c_F$ and $b_H$ depend only on the derivation components, the only new identity is obtained by taking $e_1=(X,0)$ and $e_2=(Y,0)$. The one-form component of the compatibility condition (i.) for $d_{\bT,F}$ and $\Cour{\cdot,\cdot}_H$ is
\begin{align} \label{eq:mixed-FH-identity} d_\Omega b_H(X,Y)+c_F([X,Y]) ={}& b_H(d_{\Der}X,Y) + (-1)^\m{X} b_H(X,d_{\Der}Y) \nonumber\\ &- (-1)^{(\m{X}+1)\m{Y}}i_Yd\,c_F(X) + (-1)^\m{X}\Lie_Xc_F(Y), 
\end{align}
and evaluating both sides on a derivation $Z$ gives $d_\Omega H(X,Y,Z)=dF(X,Y,Z)$.

\smallskip
 
Combining the steps, we conclude that  the pairing \eqref{eq:pairingstd}, the derivation \eqref{eq:partialstd}, the differential \eqref{eq:dF}, and the bracket \eqref{eq:H-twisted-bracket} define a dg Courant--Dorfman algebra of degree $0$. \end{proof}

We refer to the Courant-Dorfman structure in Prop.~\ref{prop:twisted} as the {\em $(F,H)$-twisted dg Courant-Dorfman algebra} on $\bT_\cR$.


\subsection{Degree shifts} \label{subsection:shifts}
We can also use the shifts in Prop. \ref{prop:shifts} to produce examples of 
standard and twisted dg Courant-Dorfman algebras of degree $-m$.
Fix $m\in\mathbb Z$, and let us denote
$$
\mathfrak{X}_m := \Der(\cR)[-m], \qquad \Omega_m^k:=\Omega_\cR^k[-km],
$$
and consider the dg $\cR$-module 
\[
\bT_\cR[-m] = \Der(\cR)[-m]\oplus \Omega^1_{\cR}[-m] = \mathfrak{X}_m \oplus \Omega_m^1,
\]
where we keep the $\cR$-module structure and differentials unchanged, and only shift the grading. 

Let us denote the natural shift operators by
\begin{equation}\label{eq:shift}
s: \Der(\cR)\to \Der(\cR)[-m], \qquad s_k: \Omega^k_\cR \to \Omega^k_\cR[-km]=\Omega^k_m,
\end{equation}
so $s$ is a map of degree $m$, and $s_k$ of degree $km$.

For $X=s(X_0)$, $Y=s(Y_0)$, and $\alpha=s_k(\alpha_0)$, define the following shifted Cartan operations (denoted in boldface): 
$$
\bm{[} X,Y \bm{]}:=s([X_0,Y_0]),
$$
which is a graded Lie bracket of degree $-m$ on $\mathfrak{X}_m$,
and 
\begin{itemize}
\item $\bm{i}_X: \Omega^k_m\to \Omega^{k-1}_m$, 
$$
\bm{i}_X\alpha := (-1)^{m\m{\alpha_0}} s_{k-1}(i_{X_0}\alpha_0),
$$ 
which is an operator of degree $\m{X}-2m$;
\item $\bm{\Lie}_X:\Omega^k_m\to \Omega^k_m$, 
$$
\bm{\Lie}_X\alpha = s_k(\Lie_{X_0}\alpha_0),
$$ 
which is an operator of degree $\m{X}-m$;
\item $\bm{d}:\Omega^k_m\to \Omega^{k+1}_m$, 
$$
\bm{d}\alpha = (-1)^{\m{\alpha_0}m}s_{k+1}(d\alpha_0),
$$
which is an operator of degree $m$; in particular, $(\bm{d}r)_0 = (-1)^{\m{r}m}dr$.
\end{itemize}
The Cartan relations for these operators remain valid with respect to the corresponding degrees.

Consider $F \in \Omega^{2,1}_\cR$ and $H\in \Omega_\cR^{3,0}$, and their shifted versions 
$$
\bm F:= s_2(F) \in (\Omega_m^{2})^{2m+1}, \qquad \bm H:=s_3(H)\in (\Omega_m^3)^{3m}. 
$$
Assume that $d_\Omega F=0$, $d H=0$ and $d_\Omega H=dF$, or, equivalently, in shifted notation,
\begin{equation*}\label{eq:FHclosed}
d_\Omega \bm{F}=0, \qquad \bm{d} \bm{H}=0, \qquad d_\Omega \bm{H}=(-1)^m \bm{d} \bm{F}.
\end{equation*}
Then the shift  (described in Prop.~\ref{prop:shifts}) of the $(F,H)$-twisted dg Courant-Dorfman algebra of degree $0$ on $\bT_\cR$  is a dg Courant-Dorfman algebra on $\bT_\cR[-m]$ of degree $-m$ that, in terms of the shifted Cartan operations, takes the following form: the differential on $\cE= \bT_\cR[-m]$ is
\begin{equation}\label{eq:dFtwist}
\dE(X,\alpha) = (d_\Der X , d_\Omega \alpha + (-1)^{\m{X}}\bm{i}_X \bm{F}),
\end{equation}
the pairing is 
\begin{equation}\label{eq:canpairingm}
\SP{(X,\alpha),(Y,\beta)}=\bm{i}_X\beta+(-1)^{|e_1||e_2|+m}\bm{i}_Y\alpha,
\end{equation}
for homogeneous $e_1=(X,\alpha)$, $e_2=(Y,\beta)$, the degree $m$ derivation is
\begin{equation}\label{eq:derm}
\partial r=(0,\bm{d}r),
\end{equation}
and the Leibniz bracket is
\begin{equation}\label{eq:leibm}
\Cour{(X,\alpha),(Y,\beta)}_H=(\bm{[}X,Y\bm{]},\bm{\Lie}_X\beta-(-1)^{(|e_1|-m)(|e_2|-m)}\bm{i}_Y\bm{d}\alpha + (-1)^{(\m{X}-m)\m{Y}}\bm{i}_Y\bm{i}_X \bm{H}).
\end{equation}

This is the {\em $(F,H)$-twisted dg Courant-Dorfman structure} on $\bT_\cR[-m]$.

\subsection{Nondegenerate pairing and normal form}
The next result gives a normal form for  dg Courant-Dorfman algebras on $\bT_\cR[-m]=\mathfrak{X}_m\oplus \Omega^1_m$ when the canonical pairing is nondegenerate.

Recall the shift operators $s$ and $s_k$ from \eqref{eq:shift}, and the notations
$X = s(X_0) \in \mathfrak{X}_m$ and $\alpha = s_k(\alpha_0) \in \Omega^k_m$.

\begin{proposition}\label{prop:exactdgCD}
Suppose that $\bT_\cR[-m]$ is endowed with a dg Courant-Dorfman algebra
structure of degree $-m$ whose pairing and derivation are the standard
ones,
\[
\SP{(X,\alpha),(Y,\beta)}
=
\bm{i}_X\beta
+
(-1)^{|e_1||e_2|+m}\bm{i}_Y\alpha,
\qquad
\partial r=(0,\bm d r).
\]
Assume that the pairing is nondegenerate, so $\partial$ corresponds to the anchor
$$
\rho: \bT_\cR[-m]\to \Der(\cR), \qquad \rho(X,\alpha)=X_0.
$$
Then there exist unique homogeneous forms $F\in \Omega^{2,1}_\cR$, 
$H \in \Omega_\cR^{3,0}$, satisfying \eqref{eq:FHclos}, such that 
the internal differential and bracket on $\bT_\cR[-m]$ are given by \eqref{eq:dFtwist} and \eqref{eq:leibm}. I.e, the dg Courant-Dorfman structure on $\bT_\cR[-m]$ is the standard one twisted by $(F,H)$.

\end{proposition}

\begin{proof} 
Recall the standard internal differential on $\cE= \bT_\cR[-m]$, $d_\bT(X,\alpha)
= (d_{\Der}(X),d_{\Omega}(\alpha))$. Let
\[
C:=\dE- d_\bT.
\]
Note that $C$ is an $\cR$-linear map of degree $1$. The fact that $\rho(\dE e) = d_\Der(\rho(e))$ implies that the $\mathfrak{X}_m$-component of $\dE(X,\alpha)$ is $d_{\Der}X$, so $C$ takes values in $\Omega^1_m$. 

Since both $d_\bT$ and $\dE$ are compatible with the standard pairing (condition (ii.)), $C$ satisfies
\[
0 = \SP{C(0,\alpha),(Y,0)} +
(-1)^{|\alpha|} \SP{(0,\alpha),C(Y,0)} = \SP{C(0,\alpha),(Y,0)},
\]
where we use that $C(Y,0)$ is a 1-form for the second equality. By nondegeneracy of the pairing, $C(0,\alpha)=0$, so  there is a map $c:\mathfrak{X}_m\to  \Omega_m^1$
of degree $1$ such that
\[
C(X,\alpha)=(0,c(X)).
\]
Again by condition (ii.), we have that
$$
0 = \SP{(0,c(X)),(Y,0)} +
(-1)^{|X|}\SP{(X,0),(0,c(Y))}
$$
for all $X$, $Y$ in $\mathfrak{X}_m$. Letting $F(X_0,Y_0):= c(X)_0(Y_0)$, we see (using the pairing and contraction conventions) that the previous identity gives
\[
F(X_0,Y_0) = -(-1)^{\m{X_0}\m{Y_0}}F(Y_0,X_0).
\]
Using the graded $\cR$-linearity of $c$ and the graded
$\cR$-linearity of the one-form $c(X)$, we conclude that
$F\in \Omega^{2,1}_{\cR}$. Using  the shifted contraction convention, one checks that
$$
c(X)_0(Y_0)=F(X_0,Y_0)=(-1)^{\m{X}}(\bm{i}_X \bm F)_0(Y_0),
$$
i.e., $c(X)= (-1)^{\m{X}}\bm{i}_X \bm F$. This proves that $\dE$ is given by \eqref{eq:dFtwist}. The identity $\dE^2=0$ is equivalent
to
$d_{\Omega}c(X)+c(d_{\Der}X)=0$, which, upon  evaluation on $Y_0$, gives precisely
\[
d_{\Omega}F=0.
\]

As for the Leibniz bracket, the fact that $\rho$ preserves brackets implies that
$$
[(X,0),(Y,0)]=([X,Y], B(X,Y)),
$$
for some $B(X,Y)\in \Omega^1_m$
Applying the Courant-Dorfman axioms (1) and (2) to pure derivations and pure 1-forms give
\begin{align*}
[(X,0),(0,\beta)] & = (0,\bm{\Lie}_X\beta),\\
[(0,\alpha),(Y,0)] &= (
0, -(-1)^{(|\alpha|-m)(|Y|-m)}
\bm{i}_Y\bm{d}\alpha),\\
[(0,\alpha),(0,\beta)]&=0.
\end{align*}
Axiom (2) implies that $B(X,Y)
= -(-1)^{(\m{X}-m)(\m{Y}-m)}B(Y,X)$,
axiom (4) gives the $\cR$-bilinearity of $B$,
and (1) applied to three pure derivations
implies the remaining graded skew-symmetry. It follows that
\[
H(X_0,Y_0,Z_0)
:=
\bigl(B(X,Y)\bigr)_0(Z_0)
\]
defines a unique $H\in \Omega^{3,0}_{\cR}$. Note that
$$
\bigl(\bm{i}_Y\bm{i}_X \bm H\bigr)_0(Z_0)
=
(-1)^{m|X_0|+|X_0||Y_0|}
H(X_0,Y_0,Z_0)
=
(-1)^{(|X|-m)|Y|}
H(X_0,Y_0,Z_0),
$$
which shows that $[\cdot,\cdot]= \Cour{\cdot,\cdot}_H$ in \eqref{eq:leibm}.

Applying the Jacobi-like identity in (c) to three pure
derivations and evaluating its one-form component on a fourth
derivation gives $dH=0$.

The remaining identity $d_\Omega H = d F$ in \eqref{eq:FHclos} is a consequence of compatibility (i.) of the bracket with the internal
differential. 

\end{proof}



\begin{thebibliography}{1}

\bibitem{AbadCrainic}
C. Abad, M. Crainic, \emph{The Weil algebra and the Van Est isomorphism}, Ann. Inst. Fourier (Grenoble) {\bf 61} (2011), no. 3, 927--970.

\bibitem{AlekseevStrobl}
A. Alekseev, T. Strobl,
{\em Current algebras and differential geometry},
J. High Energy Phys. \textbf{03} (2005), 035.

\bibitem{Poncin}
M. Ammar, N. Poncin, {\em Coalgebraic Approach to the Loday Infinity Category, Stem Differential for 2n-ary Graded and Homotopy Algebras}, Annales de l'Institut Fourier, {\bf 60} (2010), 355--387.

\bibitem{BaragliaHekmati}
D. Baraglia, P. Hekmati,
{\em Transitive Courant algebroids, string structures and T-duality},
Adv. Theor. Math. Phys. \textbf{19} (2015), no.~3,
613--672

\bibitem{BV} I.A. Batalin, G.A. Vilkovisky, \emph{Relativistic S-matrix of dynamical systems with bosons and fermions}, Physics Letters \textbf{69B} (1977), 309--312.


\bibitem{BF} I.A. Batalin, E.S. Fradkin, \emph{A generalized canonical formalism and quantization of reducible gauge theories}, Physics Letter \textbf{122B} 1983, 157--164.

\bibitem{BRS} C. Becchi, A. Rouet, R. Stora, {\em Renormalization of the abelian Higgs-Kibble model,} Commun. Math. Phys. \textbf{42} (1975), 127-162.

\bibitem{AleZabzine}
F. Bonechi, A. Cabrera, M. Zabzine,
\emph{AKSZ construction from reduction data},
J. High Energy Phys. (2012), no.~7, 068.

\bibitem{Bressler} P. Bressler, A. Chervov. \emph{Courant Algebroids}, Journal of Mathematical Sciences \textbf{128} (2005), no.~4, 3030--3053.

\bibitem{GGGR}
H. Bursztyn, A. S. Cattaneo, R. A. Mehta, M. Zambon,
\emph{Graded geometry and generalized reduction},
to appear in \emph{Mem. Eur. Math. Soc.};
arXiv:2306.01508.



\bibitem{Bursztyn}
H. Bursztyn, G. R. Cavalcanti, M. Gualtieri,
\emph{Reduction of Courant algebroids and generalized complex structures},
Adv. Math. \textbf{211} (2007), no.~2, 726--765.

\bibitem{Bursztyn1}
H. Bursztyn, G. R. Cavalcanti, M. Gualtieri,
\emph{Generalized K\"ahler and hyper-K\"ahler quotients},
in: G. Dito, J.-H. Lu, Y. Maeda, A. Weinstein (eds.),
\emph{Poisson Geometry in Mathematics and Physics},
Contemp. Math. \textbf{450},
Amer. Math. Soc., Providence, RI, 2008, 61--77.

\bibitem{thesis} P. H. Carvalho, \emph{Homological reduction of Poisson and Courant structures}, Ph.D. thesis, IMPA, 2024.

\bibitem{Carvalho}
P. H. Carvalho,
\emph{On homological reduction of Poisson structures},
Commun. Math. Phys. \textbf{406} (2025), article no.~58.

\bibitem{CDSZ}
M. Cueca, F. Dorsch, R. Sjamaar, C. Zhu, {\em The shifted symplectic geometry of derived higher groupoids}. ArXiv:2607.17362.

\bibitem{MJ}
M. Cueca, J. Schnitzer,\emph{Deformations of Lagrangian NQ-submanifolds},
Adv. Math. \textbf{458} (2024), Part A, 109952.




\bibitem{Fairon}
M. Fairon,
Introduction to graded geometry,
{\em European Journal of Mathematics} \textbf{3} (2017), no.~2,
208--222.

\bibitem{GarciaFernandezRicci}
M. Garcia-Fernandez,
{\em Ricci flow, Killing spinors, and T-duality in generalized geometry},
Adv. Math. \textbf{350} (2019), 1059--1108.


\bibitem{Gualtieri GC} M. Gualtieri, \emph{Generalized complex geometry}, Annals of Mathematics (2), \textbf{174}(1) (2011), 75-123.

\bibitem{Kalkman}
J. Kalkman,
\emph{BRST model for equivariant cohomology and representatives for the equivariant Thom class},
Commun. Math. Phys. \textbf{153} (1993), no.~3, 447--463.



\bibitem{Waldmann}
F. Keller, S. Waldmann,
\emph{Deformation theory of Courant algebroids via the Rothstein algebra},
J. Pure Appl. Algebra \textbf{219} (2015), no.~8, 3391--3426.

\bibitem{Kimura}
T. Kimura,
\emph{Generalized classical BRST cohomology and reduction of Poisson manifolds},
Commun. Math. Phys. \textbf{151} (1993), 155--182.

\bibitem{Yvette - derived bracket}
Y. Kosmann-Schwarzbach,
\emph{From Poisson algebras to Gerstenhaber algebras},
Ann. Inst. Fourier (Grenoble) \textbf{46} (1996), no.~5, 1243--1274.

\bibitem{Yvette2}
Y. Kosmann-Schwarzbach,
\emph{Derived brackets},
Lett. Math. Phys. \textbf{69} (2004), 61--87.


\bibitem{Kostant}
B. Kostant, S. Sternberg,
\emph{Symplectic reduction, BRS cohomology, and infinite-dimensional Clifford algebras},
Ann. Physics \textbf{176} (1987), no.~1, 49--113.


\bibitem{KotovSalnikov}
A. Kotov, V. Salnikov,
The category of $\mathbb{Z}$-graded manifolds: What happens if you do not stay positive,
{\em Differential Geometry and its Applications} \textbf{93} (2024),
Paper No.~102109, 25 pp.

\bibitem{LWX}
Z.-J. Liu, A. Weinstein, P. Xu,
\emph{Manin triples for Lie bialgebroids},
J. Differential Geom. \textbf{45} (1997), no.~3, 547--574.

\bibitem{Mandal}
A. Mandal,
\emph{On exact Courant algebras},
Comm. Algebra \textbf{44} (2016), no.~5, 2058--2066.

\bibitem{Mehta's Thesis}
R. A. Mehta,
\emph{Supergroupoids, double structures, and equivariant cohomology},
Ph.D. thesis, University of California, Berkeley, 2006;
arXiv/0605356.


\bibitem{Pym-Safronov} B. Pym, P. Safronov, \emph{Shifted symplectic Lie algebroids}, Int. Math. Res. Not. \textbf{21} (2020), pp. 7489-7557.

\bibitem{RoytenbergThesis}
D. Roytenberg,
\emph{Courant algebroids, derived brackets and even symplectic supermanifolds},
Ph.D. thesis, University of California, Berkeley, 1999;
arXiv/9910078.

\bibitem{Roytenberg}
D. Roytenberg,
\emph{On the structure of graded symplectic supermanifolds and Courant algebroids},
in: T. Voronov (ed.),
\emph{Quantization, Poisson Brackets and Beyond},
Contemp. Math. \textbf{315},
Amer. Math. Soc., Providence, RI, 2002, 169--185.

\bibitem{Roytenberg - AKSZ} D. Roytenberg,  \emph{AKSZ-BV formalism and Courant algebroid-induced topological field theories}, Lett. Math. Phys. {\bf 79} (2007),  143-159. 

\bibitem{Roytenberg CD}
D. Roytenberg,
\emph{Courant--Dorfman algebras and their cohomology},
Lett. Math. Phys. \textbf{90} (2009), 311--351.

\bibitem{SeveraValachSupergravity}
P. \v Severa, F. Valach,
{\em Courant algebroids, Poisson--Lie T-duality, and type II supergravities},
Commun. Math. Phys. \textbf{375} (2020), 307--344.

\bibitem{SeWe}
P. \v{S}evera, A. Weinstein,
\emph{Poisson geometry with a 3-form background},
Progr. Theoret. Phys. Suppl. \textbf{144} (2001), 145--154.

\bibitem{Nikolay}
N. Sheshko, {\em Derived Symplectic Reduction in Differential Geometry}. ArXiv:2605.16226.


\bibitem{Stasheff}
J. Stasheff,
\emph{Homological reduction of constrained Poisson algebras},
J. Differential Geom. \textbf{45} (1997), no.~1, 221--240.

\bibitem{Uchino}
K. Uchino, \emph{Derived brackets and sh Leibniz algebras},
Journal of Pure and Applied Algebra {\bf 215} (2011), 1102--1111.

\bibitem{Voronov - Higher brackets}
T. Voronov,
\emph{Higher derived brackets and homotopy algebras},
J. Pure Appl. Algebra \textbf{202} (2005), nos.~1--3, 133--153.

\bibitem{VoronovGradedGeometry}
T. Voronov,
{\em Graded geometry, $Q$-manifolds, and microformal geometry},
Fortschritte der Physik \textbf{67} (2019), nos.~8--9,
Article No.~1910023.


\bibitem{Vysoky}
J. Vysoky,
\emph{Graded generalized geometry},
J. Geom. Phys. \textbf{182} (2022), 104683.




\end{thebibliography}
\end{document}